\documentclass[11pt,reqno]{amsart}
\usepackage{amsthm,amsmath,amssymb}

\usepackage{graphicx}

\usepackage[left=1in, right=1in, top=1in, bottom=1in]{geometry}

\newcommand{\ds}{\displaystyle}

\newcommand{\B}{\tilde{B}}
\newcommand{\Ht}{\tilde{H}}
\newcommand{\Jt}{\tilde{J}}
\newcommand{\modulo}{\,(\text{mod } 2)}

\theoremstyle{definition}
\newtheorem{thm}{Theorem}[section]

\newtheorem{prop}[thm]{Proposition}

\newtheorem{rem}[thm]{Remark}

\numberwithin{equation}{section}

\title[Ghost series and Andrews-Bressoud identities] {Ghost series and
a motivated proof of the Andrews-Bressoud identities}

\author{Shashank Kanade, James Lepowsky, Matthew C. Russell and Andrew V. Sills}
\thanks{The work of the fourth author is supported in part by
National Security Agency Grant P13217-39G3217.}

\allowdisplaybreaks

\begin{document}

\begin{abstract}
We present what we call a ``motivated proof'' of the
Andrews-Bressoud partition identities for even moduli.  A ``motivated
proof'' of the Rogers-Ramanujan identities was given by G.\ E.\ Andrews
and R.\ J.\ Baxter, and this proof was generalized to the odd-moduli
case of Gordon's identities by J.\ Lepowsky and M.\ Zhu.  Recently, a
``motivated proof'' of the somewhat analogous
G\"ollnitz-Gordon-Andrews identities has been found.  In the present
work, we introduce ``shelves'' of formal series incorporating what we
call ``ghost series,'' which allow us to pass from one shelf to the
next via natural recursions, leading to our motivated proof.  We
anticipate that these new series will provide insight into the ongoing
program of vertex-algebraic categorification of the various
``motivated proofs.''
\end{abstract}

\maketitle

\section{Introduction}

The classical Rogers-Ramanujan partition identities have numerous
generalizations in various directions, notable among which are the
generalizations by Gordon-Andrews for odd moduli, extended to even
moduli by Andrews-Bressoud.  The
product sides of the Rogers-Ramanujan identities enumerate the
partitions whose parts obey certain restrictions modulo 5, and the sum
sides enumerate the partitions with certain difference-two and initial
conditions.  Generalizations of the Rogers-Ramanujan identities for
all odd moduli were discovered by B. Gordon \cite{G} and G. E. Andrews
\cite{A1}.  Analogous identities for the even moduli of the form
$4k+2$ were discovered by Andrews in \cite{A2} and \cite{A3}, and
subsequently, for all the even moduli, by D. M. Bressoud in \cite{Br}.  The
Andrews-Bressoud identities state that for any $k\ge 2$ and
$i\in \left\{1,\dots,k\right\}$,
\begin{align*}
\frac{\prod_{m\ge 1}
(1-q^{2km})(1-q^{2km-k-i+1})(1-q^{2km-k+i-1})}{\prod_{m\ge
1}(1-q^m)} = \sum\limits_{n\ge 0} b_{k,i}(n)q^n,
\end{align*}
where $b_{k,i}(n)$ is the number of partitions
$\pi=(\pi_1,\dots,\pi_{s})$ of $n$ (with $\pi_t\geq \pi_{t+1}$) such
that
\begin{enumerate}
\item $\pi_{t} - \pi_{t+k-1} \geq 2$, \item $\pi_t - \pi_{t+k-2}
\leq 1$ only if $\pi_t + \pi_{t+1} + \cdots + \pi_{t+k-2} \equiv
i+k \pmod{2}$,
\item at most $k-i$ parts of $\pi$ equal 1.
\end{enumerate}
Here we have replaced $r$ by $k-i+1$ in the statement of the main
theorem of \cite{Br}. Also, here and below, $q$ is a formal variable.

The product side above is the generating function for
partitions not congruent to $0$ or $\pm(k-i+1)$ modulo $2k$, 
except in the case $i=1$. 
The statement of the main theorem in \cite{Br} excluded this exceptional case,
simply because no natural combinatorial
interpretation of the corresponding product side was known at the
time, but the proof of the main theorem in \cite{Br}, in particular,
Lemma 3, certainly did cover this case.  Building on work of Andrews-Lewis
\cite{AL}, an elegant combinatorial interpretation of the product in
the case $i=1$ was discovered in \cite{S}.

As we will recall below, the Gordon-Andrews-Bressoud identities, as
well as more general families of such identities, also arise very
naturally from the representation theory of vertex operator algebras.
In \cite{LW2}--\cite{LW4}, the Rogers-Ramanujan identities were
proved, and the more general Gordon-Andrews-Bressoud identities were
interpreted (including the case $i=1$ for even moduli), using the
theory of $Z$-algebras, which were invented for this very purpose.  In
\cite{MP}, this interpretation of the more general
Gordon-Andrews-Bressoud identities was strengthened to a
vertex-operator-theoretic proof of these identities.  The $Z$-algebra
structures came to be understood in retrospect as the natural
generating substructures of certain twisted modules for certain
generalized vertex operator algebras, once the theory of generalized
vertex operator algebras and their modules and twisted modules was
later developed.

As is recalled in the Introduction of \cite{CKLMQRS}, the
Rogers-Ramanujan identities also arose in R. J. Baxter's work
\cite{Ba} in statistical mechanics.  Since then, advances in this area
have been made by a large number of authors, including A. Berkovich,
P. J. Forrester, B. McCoy, A. Schilling and S. O. Warnaar among many
others; see the references in \cite{GOW} for a variety of relevant
works.  In \cite{GOW}, M. Griffin, K. Ono and Warnaar have illuminated
certain arithmetic properties of Rogers-Ramanujan-type expressions
using a framework incorporating the Hall-Littlewood polynomials.

In the present paper we shall present what we call a ``motivated
proof'' of the Andrews-Bressoud even-modulus identities.  The first
``motivated proof'' in this spirit was given by G. E. Andrews and
Baxter for the Rogers-Ramanujan identities in \cite{AB}, and
that proof was observed in that paper to be essentially the same as a
Rogers-Ramanujan proof in \cite{RR} and as Baxter's proof in
\cite{Ba}. This proof was generalized to the case of the
Gordon-Andrews identities in \cite{LZ} and recently, a ``motivated proof''
of the somewhat analogous G\"ollnitz-Gordon-Andrews identities was
presented in \cite{CKLMQRS}.

However, the problem, solved in this paper, of constructing an
analogous ``motivated proof'' of the even-modulus Andrews-Bressoud
identities turned out to be much more subtle than for the
Gordon-Andrews and G\"ollnitz-Gordon-Andrews identities.  Interesting
new phenomena emerged, as we shall see in a moment.

For discussions of the structure of ``motivated proofs''
(this designation has now become a technical term in our work,
and we shall drop the quotation marks) and relevant
terminology, and of the likely importance of such proofs from the
standpoint of vertex operator algebra theory, we refer the reader to the
Introductions in \cite{LZ} and \cite{CKLMQRS}, which also include some
references touching briefly on some of the large amount of work that
has been done relating partition identities, old and new, with vertex
operator algebra theory.  This importance can be summarized as
follows: It is expected that vertex-algebraic ``categorifications'' of
the steps in the motivated proofs will provide useful insights into
the representation theory of vertex operator algebras and will also
facilitate the discovery of futher families of natural identities in
this spirit.

As is explained in the Introduction of \cite{CKLMQRS}, motivated
proofs, as now understood, proceed by the creation of successively
higher ``shelves'' of formal power series in a formal variable $q$
(involving only nonnegative integral powers of $q$).
For a fixed integer $k$, the given product expressions form the
0$^\text{th}$ shelf, and each new shelf is formed by taking linear
combinations of the $q$-series on the previous shelf, over the field
of fractions of the ring of polynomials in the formal variable $q$.
The exact formulas for these linear combinations constitute a crucial
ingredient of the motivated proofs.  The main point is that for each
constructed series, when the leading term 1 is subtracted, the result
is a series divisible by a high power of $q$; this mechanism was
called the ``Empirical Hypothesis'' in \cite{AB}, and in all these
motivated proofs, such an Empirical Hypothesis readily leads to a
proof of the desired identities.  As an example, consider the
simplest nontrivial case---that of the Rogers-Ramanujan identities: Here,
the 0$^\text{th}$ shelf is formed by the two Rogers-Ramanujan
products, say, $G_1$ and $G_2$, and the two $q$-series forming the
$1^\text{st}$ shelf are a copy of $G_2$ itself and
$G_3=(G_1-G_2)/q$, which indeed turns out to be a formal power series
in $q$.  The higher shelves are then formed analogously.
The motivated proof hinges on the Empirical Hypothesis that the 
$j^\text{th}$ series $G_j$, again a formal power series, is such that
$G_j - 1$ is divisible by $q^j$.

The source of the deep connection with the representation theory of
vertex operator algebras is that the products $G_1$ and $G_2$ are both
the graded dimensions (i.e., the generating functions for the dimensions
of the finite-dimensional homogeneous subspaces)
of certain twisted modules for a certain
generalized operator vertex operator algebra, and formulas of the type
$G_{3}=(G_1-G_2)/q$ are expected to be ``categorified'' by means of
exact sequences among these modules, where the maps between the
modules are formed using vertex-algebraic intertwining operators.  This
phenomenon is expected to arise in considerable generality.
The broad program of such
categorifications was initiated by J.\ Lepowsky and is an active area
of research.  That such categorifications can be found for motivated
proofs using ``twisted'' and ``relativized'' intertwining operators
was an idea of J.\ Lepowsky and A.\ Milas, and is currently under
investigation.  Similar categorifications of the Rogers-Ramanujan,
Rogers-Selberg and Euler recursions, as well as of new
recursions, using untwisted and unrelativized intertwining operators,
have already been successfully carried out in \cite{CLM1}, \cite{CLM2}, 
\cite{Cal1}, \cite{Cal2}, \cite{CalLM1}--\cite{CalLM4} and \cite{Sa}.  

The products in the Gordon-Andrews identities are the graded
dimensions of twisted modules for certain generalized vertex operator
algebras corresponding to the affine Lie algebra $A_1^{(1)}$ at odd
levels, and the Andrews-Bressoud identities similarly arise from the
even levels (see \cite{LM} --- where the vertex-algebraic structure had not
yet been discovered --- and \cite{LW1}--\cite{LW4}).  Lepowsky-Wilson's
$Z$-algebraic interpretation of the sum sides of these identities
treats all the levels (and hence both families of identities) on an
equal footing (\cite{LW1}--\cite{LW4}).  Correspondingly, having
motivated proofs of both families of identities, not just of the
odd-modulus identities, would provide important insight.  For the full
family of Gordon-Andrews identities, the aforementioned linear
combinations have the same ``shape'' as for the Rogers-Ramanujan
special case.  Namely, they involve appropriate subtractions in the
numerator and a pure power of $q$ in the denominator.  However, the
case of the even-modulus Andrews-Bressoud identities is starkly
different, in that the denominator now involves more complicated
expressions such as a sum of two pure powers of $q$.  The source of
this substantial subtlety is the parity conditions on the sum sides of
the identities (see (2) above).  Such a division is expected to be
difficult to categorify (and is not ``motivated,'' either), and the
problem now was to try to invent a natural mechanism that would
restore one's ability to divide by pure powers of $q$, and to thereby
arrive at a suitable ``Empirical Hypothesis,'' for the even moduli.

We achieve this in the present paper by introducing shelves of what we
call ``ghost series.''  Given the $j^\text{th}$ shelf of ``official''
Bressoud series, we introduce relations which simultaneously define a
$j^\text{th}$ shelf of ``ghost series'' and also facilitate the
passage to a $(j+1)^\text{st}$ shelf using pure powers of $q$ dividing
the subtractions of the official and the ghost series.  In this way,
we are able to obtain an Empirical Hypothesis that yields the
Andrews-Bressoud identities.  After our main theorem establishing
these identities has been proved, we present a combinatorial
interpretation of the ghost series, and it turns out that they differ
{}from the official series only in the parity condition.

In addition to the appearance of ghost series, there is another
significant way in which this work differs from the earlier works
\cite{AB}, \cite{LZ} and \cite{CKLMQRS}.  In those works, the
recursions defining the successively higher shelves could be, or could
have been, predicted by appropriately specializing the already-known
$(a,x,q)$-recursions, proved by Andrews in Chapter 7 of \cite{A3},
satisfied by the corresponding Rogers-Andrews expressions given in
(7.2.1) and (7.2.2) of \cite{A3}, and the recursions 
for proving the G\"ollnitz-Gordon-Andrews identities were indeed
``discovered'' in this way in \cite{CKLMQRS} (cf. Appendices A and C of
\cite{CKLMQRS}).  But in the case of the Andrews-Bressoud identities,
unlike in \cite{CKLMQRS}, the interplay between the various formal
power series, including the ghost series, was discovered by purely
empirical means.  In particular, the ``Empirical Hypothesis'' in the
present paper was {\it indeed} empirical; in \cite{CKLMQRS}, the
``Empirical Hypothesis'' was ``secretly known in advance'' because of
its source in Andrews's just-mentioned $(a,x,q)$-recursions. 
Chronologically, the decision was first made to examine these ghost series 
(with the ``wrong'' parity conditions), using the combinatorial 
interpretations stated in Section 7, in our search for a proper 
``Empirical Hypothesis.'' 
Then, the relations which enable one to simultaneously
define the ghost series and move to the next shelf via a division by a
pure power of $q$ were experimentally determined.
It is also interesting to note that the proof of our main theorem requires
the Empirical Hypothesis only for the official series and not for the
ghost series.

The discovery of ghost series raises many interesting
questions.  These series have not yet been observed to come from
vertex-algebraic structures, but they seem extremely natural from the point
of view of the categorification mentioned above.  Therefore, it is natural
to ask where they reside in the theory of vertex operator algebras.  This
question is being examined.  Their modularity properties and
representations as multisums and weighted $q$-sums of infinite
products are under investigation as well.  It is hoped that the new
insight of ghost series will prove fruitful in providing motivated
proofs of other Rogers-Ramanujan-style partition identities, and such
a program is also underway.

This paper is organized as follows: In Section \ref{sec:setting} we
recall basic notations regarding $q$-series and partitions.  Then,
after using the Jacobi triple product identity to re-express the
Andrews-Bressoud products in the usual more amenable form, we present
our recursions for building the official and ghost series comprising
the higher shelves.  In Section \ref{sec:0ghosts}, we derive
closed-form expressions for the 0$^\text{th}$-shelf ghosts. These
expressions are used as the base case of the induction in the proof of
Theorem \ref{thm:main} in Section \ref{sec:closedforms}, giving
closed-form expressions for the official series and the ghost series
on all the shelves.  In Section \ref{sec:EH}, using Theorem
\ref{thm:main}, we derive our Empirical Hypothesis in a number of
different strengths; the weakest of these is enough to prove our main
theorem.  We put various recursions and relations into an elegant
matrix formulation in Section \ref{sec:matrix}, and as a consequence
we derive one form of our main theorem, Theorem
\ref{thm:infty-main}. In Section \ref{sec:comb} we complete our
motivated proof of the Andrews-Bressoud identities.  Section
\ref{sec:(x,q)dict} is concerned with comparing our closed-form
expressions with various $(x,q)$-series contained in the paper
\cite{CoLoMa} of S. Corteel, J. Lovejoy and O. Mallet --- 
we provide a ``dictionary'' between our closed-form
expressions and appropriate specializations of the formal series
$\tilde{J}$ appearing in \cite{CoLoMa}.  Finally, in Section
\ref{sec:(x,q)forghosts}, we ``reverse-engineer'' our method and
appropriately replace certain powers of $q$ with $x$ in our
closed-form expressions to arrive at (new) $(x,q)$-expressions that govern
the ghosts.

Just as in \cite{LZ} and \cite{CKLMQRS}, throughout this paper we
treat power series as purely formal series rather than as convergent
series (in suitable domains) in complex variables.

\section{The setting}
\label{sec:setting}

First we recall the standard notation
\begin{align}
(a)_n &= \prod_{s=0}^{n-1} (1-aq^s)\label{dfn:a_n}\\
(q)_n &= \prod_{s=1}^{n} (1-q^s)\label{dfn:q_n}\\
(q)_\infty &= \prod_{s \ge 1} (1-q^s)\label{dfn:q_infty},
\end{align}
where $a$ is any formal variable or complex number.

Fix an integer $k\ge 2$. For each $i\in\left\{1,\dots,k\right\}$, define
\begin{align}
\label{def:B_i0thshelf}
B_i=\dfrac{\prod_{m\ge 1}(1-q^{2km})
(1-q^{2km-k+i-1})(1-q^{2km-k-i+1})}{(q)_\infty}.
\end{align}
Recalling the Jacobi triple product identity, 
\begin{eqnarray*}
\sum_{\lambda \in \mathbb{Z}} (-1)^\lambda z^\lambda q^{\lambda^2} 
& = &
\prod_{m \geq 0} (1-q^{2m+2})(1- z q^{2m+1})(1-z^{-1} q^{2m+1}),
\end{eqnarray*}
and replacing $q$ by $q^k$ and $z$ by
$q^{i-1}$, we have
\begin{align}
B_i 
&= \frac{\sum_{n \in \mathbb{Z}} 
(-1)^n q^{n(i-1)} q^{kn^2}}{(q)_\infty}\nonumber\\
&= \frac{\sum_{n \ge 0} (-1)^n q^{kn^2+n(i-1)}
(1-q^{(k-i+1)(2n+1)})}{(q)_\infty}.
\label{eqn:B_iexplicit}
\end{align}

We add to our list the next $k-1$ ``official'' series $B_l$ for
$l\in\left\{k+1, k+2, \dots, 2k-1\right\}$,
along with what we shall call ``ghost series'' counterparts $\B_l$
for $l\in\left\{2,3,\dots, k\right\}$, via the relations
\begin{equation}\label{edgematching}
B_{k+1} = \frac{B_{k-1} - \B_k }{q} = \B_k ,
\end{equation} and
\begin{equation}\label{shelf1}
B_{k+h} =  
\frac{B_{k - h } - \B_{ k-h+1 } } { q^{ h } }
= \frac{\B_{ k - h + 1 } - B_{ k - h + 2} }{ q^{ h - 1} }  
\end{equation} for $h\in\left\{2,3,\dots, k-1\right\}$.

\begin{rem}\label{rem:1stghost}
Note that we have not defined $\B_1$. 
Indeed, $\B_1$ will not be necessary in proving our main theorem.
However, once we get the combinatorial interpretation of
the ghosts in Theorem \ref{thm:comb-main-ghosts}, we
will be able to attach meaning to $\B_1$; see Remark
\ref{rem:1stghostmeaning}.
\end{rem}

\begin{rem}\label{rem:relations}
The right-hand equalities in equations \eqref{edgematching}
and \eqref{shelf1} serve to define the zeroth-shelf
ghosts. After these ghosts are defined, 
\eqref{edgematching} and \eqref{shelf1} give the
definition of the official series on the first shelf.
\end{rem}

Continuing similarly, for $j\geq 0$
we define ghosts and official series on all shelves by
\begin{equation}\label{def:higherjedge}
B_{(k-1)(j+1)+2} = \frac{B_{(k-1)j+k-1} - \B_{ (k-1)j+k} }{q^{j+1}}
 = \B_{(k-1)j + k},
\end{equation} and
\begin{equation}\label{def:higherj}
B_{(k-1)(j+1)+i} =  
\frac{B_{(k-1)j + k-i+1 } - \B_{ (k-1)j +k-i+2 } } { q^{ (j+1) (i-1) } }
= \frac{\B_{ (k-1)j + k-i+2 } - B_{ (k-1)j+k-i+3 } }{ q^{ (j+1)( i - 2) } }  
\end{equation} 
for $i\in\left\{3,4,\dots, k\right\}$,
with Remark \ref{rem:relations} extending to all $j$.

\begin{rem}\label{rem:crux}
As previewed in the Introduction, one clearly sees 
that ghost series yield the ``correctly'' shaped relations 
\eqref{def:higherjedge} and \eqref{def:higherj}, i.e., 
these relations involve pure powers of $q$ dividing
the subtractions of the appropriate series.
\end{rem}

As mentioned above in Remark \ref{rem:relations}, 
using the right-hand equalities in equations \eqref{def:higherjedge} and 
\eqref{def:higherj}, one can define the ghosts explicitly for $j\geq 0$:
\begin{align}
\B_{(k-1)j+i}&= \frac{B_{(k-1)j+i-1}+q^{j+1}B_{(k-1)j+i+1}}{1+q^{j+1}}
\qquad \text{for } i\in\{2,\dots,k-1\}\label{def:ghost_others},\\
\B_{(k-1)j+k}&=\frac{B_{(k-1)j+k-1}}{1+q^{j+1}} \label{def:ghost_k}.
\end{align}

\begin{rem}\label{rem:preview_closed_form}
In Theorem \ref{thm:main}, we will provide closed-form expressions for the
various official and ghost series. In order to obtain the closed-form
expressions for the official series, we shall use the left-hand equalities
in \eqref{def:higherjedge} and \eqref{def:higherj}, and for the ghosts, 
we will use \eqref{def:ghost_others} and \eqref{def:ghost_k}.
\end{rem}

\begin{rem}
The case $k=2$ corresponds to a pair of Euler identities. 
In this case, several relations, namely, \eqref{shelf1}, \eqref{def:higherj}
and \eqref{def:ghost_others}, become vacuous, and only the ``edge-matching'' 
relations, namely, \eqref{edgematching}, \eqref{def:higherjedge} and 
\eqref{def:ghost_k}, survive. 
As a result of this, the official series on the various shelves 
also manifest themselves as ghost series.
\end{rem}

Now we turn to our notation regarding partitions.
As usual, a \emph{partition} $\pi$ \emph{of} $n$ is a finite nonincreasing
sequence of positive integers $(\pi_1, \pi_2, \dots, \pi_t)$ such that
$\pi_1 + \pi_2 + \cdots + \pi_t = n$.  Each $\pi_s$ is called a
\emph{part} of $\pi$.  The \emph{length} $\ell(\pi)$ of $\pi$ is the
number of parts in $\pi$.  For any positive integer $p$, the
\emph{multiplicity} of $p$ in $\pi$, denoted $m_p(\pi)$, is the number
of parts in $\pi$ equal to $p$.

We will prove two analogues of~\cite[Theorem 4.1]{LZ} --- for the
official series (recovering the Andrews-Bressoud identities as a
special case) and for the ghosts: For $r = (k-1)j +
i$, with  $j\geq 0$ and  $1\leq i \leq k$,
\[ B_r = \sum_{n\geq 0} b_{k,r}(n) q^n, \]
where $b_{k,r}(n)$ denotes the number of partitions 
$\pi = (\pi_1, \pi_2, \dots, \pi_{\ell(\pi)})$ of $n$ such that 
\begin{enumerate}
\item $\pi_{t} - \pi_{t+k-1} \geq 2$,
\item $\pi_{\ell(\pi)} \geq j+1$,
\item $m_{j+1}(\pi) \leq k-i$,
\item $\pi_t - \pi_{t+k-2} \leq 1$ only if 
$\pi_t + \pi_{t+1} + \cdots + \pi_{t+k-2} \equiv 
r+k \pmod{2}$.
\end{enumerate}
Moreover,  for $r = (k-1)j + i$, with $j\geq 0$ and $2\leq i \leq k$,
\[ \B_r = \sum_{n\geq 0} \tilde{b}_{k,r}(n) q^n, \]
where $\tilde{b}_{k,r}(n)$ denotes the number of partitions 
$\pi = (\pi_1, \pi_2, \dots, \pi_{\ell(\pi)})$ of $n$ such that 
\begin{enumerate}
\item $\pi_{t} - \pi_{t+k-1} \geq 2$,
\item $\pi_{\ell(\pi)} \geq j+1$,
\item $m_{j+1}(\pi) \leq k-i$,
\item $\pi_t - \pi_{t+k-2} \leq 1$ only if 
$\pi_t + \pi_{t+1} + \cdots + \pi_{t+k-2} \equiv 
r+k+1 \pmod{2}$.
\end{enumerate}
Note that for the ghosts, the parity condition is the ``wrong'' one;
this is the key idea.

\section{The zeroth shelf}
\label{sec:0ghosts}
%%%%%%%%%%%%%%
In this section we derive closed-form expressions for the ghosts on the
zeroth shelf.  These closed-form expressions will be used
later as the base case of the induction in the proof of Theorem
\ref{thm:main}.  Recall equation \eqref{eqn:B_iexplicit}: For
$i\in\left\{1,\dots,k\right\}$,
\begin{align}
B_i &= \ds\sum\limits_{n\geq 0}(-1)^n
\frac{q^{kn^2+n(i-1)}(1-q^{(k-i+1)(2n+1)})  }{(q)_\infty}.
\end{align} 

Solving \eqref{edgematching} for $\B_k$,
we get that
\begin{equation}
\B_k=\frac{B_{k-1}}{1+q}.\label{eqn:ghosts_k}
\end{equation}
Next, solving \eqref{shelf1} for $\B_{k-i+1}$,
\begin{equation}\label{eqn:ghosts_k-i+1}
\B_{k-i+1}=\ds\frac{B_{k-i} + qB_{k-i+2}}{1+q}
\end{equation}
for $i\in\left\{2,3,\dots,k-1\right\}$.  Equations \eqref{eqn:ghosts_k} and
\eqref{eqn:ghosts_k-i+1} express the ghosts on the zeroth shelf
in terms of the official series on the zeroth shelf.
Re-indexing \eqref{eqn:ghosts_k-i+1}, we get,
for $i\in\left\{2,3,\dots,k-1\right\}$,
\begin{equation}\label{ghosts_i}
\B_{i}=\ds\frac{B_{i-1} + qB_{i+1}}{1+q}.
\end{equation}

{}From \eqref{eqn:B_iexplicit} (as recalled above), we gather that for
$i\in\left\{2,3,\dots,k-1\right\}$,
\begin{align}
B_{i-1} &= \ds\sum\limits_{n\geq
0}(-1)^n\frac{q^{kn^2+n(i-2)}(1-q^{(k-i+2)(2n+1)}) }{(q)_\infty},\\
B_{i+1} &= \ds\sum\limits_{n\geq
0}(-1)^n\frac{q^{kn^2+ni}(1-q^{(k-i)(2n+1)}) }{(q)_\infty}.
\end{align}
Hence
\begin{align*}
(q)_\infty( B_{i-1} + q B_{i+1} )\\
&=\ds\sum\limits_{n\geq 0} (-1)^n
\left(q^{kn^2+n(i-2)}\left(1-q^{(k-i+2)(2n+1)}\right)
+ {q^{kn^2+ni+1}\left(1-q^{(k-i)(2n+1)}\right)  }\right)  \\
&=\ds\sum\limits_{n\geq 0} (-1)^n q^{kn^2+n(i-2)}\left(1-q^{(k-i+2)(2n+1)}
+ {q^{2n+1}\left(1-q^{(k-i)(2n+1)}\right)  }\right)  \\
&=\ds\sum\limits_{n\geq 0} (-1)^n
q^{kn^2+n(i-2)}\left(1-q^{(k-i+1)(2n+1)+2n+1}+
{q^{2n+1}-q^{(k-i+1)(2n+1)} }\right) \\
&=\ds\sum\limits_{n\geq 0} (-1)^n
q^{kn^2+n(i-2)}\left(1-q^{(k-i+1)(2n+1)}\right)(1+q^{2n+1}).
\end{align*}
Therefore, for $i\in\left\{2,3,\dots,k-1\right\}$,
{}from~\eqref{ghosts_i},
\begin{equation}\label{ghostsformula}
\B_i = \ds\frac{\ds{B_{i-1} + q B_{i+1}}}{1+q}= \ds\sum\limits_{n\geq
0}(-1)^n\frac{q^{kn^2+n(i-2)}(1-q^{(k-i+1)(2n+1)})(1+q^{2n+1})}
{(q)_\infty(1+q)}.
\end{equation}

Now we find $\B_k$ in the same way:
\[
\B_k=\frac{B_{k-1}}{1+q}=\ds\sum\limits_{n\geq
0}(-1)^n\frac{q^{kn^2+n(k-2)}(1-q^{2(2n+1)})}{(q)_\infty(1+q)},
\]
and it is easy to see that this is the same expression we get by
setting $i=k$ in \eqref{ghostsformula}.  Therefore,
\eqref{ghostsformula} holds for $i\in\left\{2,3,\dots,k\right\}$.

%%%%%%%%%%%%

\section{Determination of closed-form expressions for higher shelves and ghosts}
\label{sec:closedforms}

The aim of this section is to provide explicit closed-form expressions
for all of the ghosts and the official series. 
These closed-form expressions will enable us to provide a proof of 
our Empirical Hypothesis in Section \ref{sec:EH}.

\begin{thm}\label{thm:main}
For $j\ge 0$ and $i\in\left\{1,2,\dots,k\right\}$,
\[
B_{(k-1)j+i} \in \mathbb{C}[[q]]
\]
and in fact,
\begin{align}
\label{B_general}&B_{(k-1)j+i}\\
& =\ds\sum\limits_{n\geq
0}(-1)^n\frac{q^{kn^2+((k-1)j+i-1)n}(1-q^{(2n+j+1)(k-i+1)})
(1-q^{2(n+1)})\cdots(1-q^{2(n+j)})}{(q)_\infty(1+q)\cdots(1+q^j)}.
\nonumber
\end{align}
Denoting the right-hand side of \eqref{B_general} by $RHS_{j,i}$,
we have that for each $j\geq 1$, the two expressions for
$B_{(k-1)j+i}$ match:
\begin{equation}
RHS_{j-1,k} = RHS_{j,1} .
\label{edgematching_general}
\end{equation}
Moreover,
for $j\ge 0$ and $i\in\left\{2,\dots,k\right\}$,
\[
\B_{(k-1)j+i} \in \mathbb{C}[[q]]
\]
and in fact,
\begin{align}
\label{Bt_general}& \B_{(k-1)j+i}\\
& =\ds\sum\limits_{n\geq 0}(-1)^n
\frac{
q^{kn^2+((k-1)j+i-2)n}(1-q^{2(n+1)})\cdots(1-q^{2(n+j)})
(1+q^{2n+j+1})(1-q^{(2n+j+1)(k-i+1)})
}
{(q)_\infty(1+q)\cdots(1+q^{j+1})}.
\nonumber
\end{align}
\end{thm}

\begin{proof}
{}From the previous sections, we see that the conclusions hold if $j=0$. 
We proceed by induction. Assume that the assertions in the theorem are
true for all $i\in\left\{1,2,\dots k\right\}$ for a certain index $j$.

For $i\in\left\{2,\dots,k\right\}$,
\[
B_{(k-1)(j+1)+i}=\ds\frac{B_{(k-1)j+k-i+1}- 
\B_{(k-1)j+k-i+2}}{q^{(j+1)(i-1)}}.
\]
Therefore,
\begin{align}
&\lefteqn{(q)_\infty(1+q)\cdots(1+q^{j+1})q^{(j+1)(i-1)} B_{(k-1)(j+1)+i} }\nonumber\\
= &  
\ds\sum\limits_{n\geq 0}(-1)^n {q^{kn^2+((k-1)j+k-i)n}(1-q^{(2n+j+1)i})
(1-q^{2(n+1)})\cdots(1-q^{2(n+j)})}(1+q^{j+1})\nonumber\\
& \mbox{} - \ds\sum\limits_{n\geq 0}(-1)^n
{q^{kn^2+((k-1)j+k-i)n}(1-q^{2(n+1)})\cdots
(1-q^{2(n+j)})(1+q^{2n+j+1})(1-q^{(2n+j+1)(i-1)})}
\nonumber\\
= & \ds\sum\limits_{n\geq 0}(-1)^n
q^{kn^2+((k-1)j+k-i)n}(1-q^{2(n+1)})\cdots(1-q^{2(n+j)})\nonumber \\
& \cdot\left(1-q^{(2n+j+1)i}+q^{j+1}- q^{(2n+j+1)i + j+1}-1-q^{2n+j+1}  
+ q^{(2n+j+1)(i-1)} +q^{(2n+j+1)i}  \right)\nonumber\\
= & \ds\sum\limits_{n\geq 0}(-1)^n
q^{kn^2+((k-1)j+k-i)n}(1-q^{2(n+1)})\cdots(1-q^{2(n+j)})\nonumber \\
& \cdot\left(q^{j+1}- q^{(2n+j+1)i+ j+1}-q^{2n+j+1}+ q^{(2n+j+1)(i-1)}\right)
\nonumber\\
= & \ds\sum\limits_{n\geq 0}(-1)^n
q^{kn^2+((k-1)j+k-i)n}(1-q^{2(n+1)})\cdots(1-q^{2(n+j)})\nonumber \\
&  \left( q^{j+1}(1-q^{2n}) +  q^{(2n+j+1)(i-1)} (1-q^{2(n+j+1)}) \right)
\nonumber\\
= & \ds\sum\limits_{n\geq 0}(-1)^n
q^{kn^2+((k-1)j+k-i)n+j+1}(1-q^{2n})
(1-q^{2(n+1)})\cdots(1-q^{2(n+j)})\label{before_index_change} \\
&\mbox{} + \ds\sum\limits_{n\geq 0}(-1)^n 
 q^{kn^2+((k-1)j+k-i)n+{(2n+j+1)(i-1)}}
(1-q^{2(n+1)})\cdots (1-q^{2(n+j+1)}).\label{does_not_change}
\end{align}
In \eqref{before_index_change}, the term corresponding to $n=0$ is $0$, 
and hence, 
making the index change $n\mapsto n+1$, we get
\begin{align}
& \ds\sum\limits_{n\geq 0}(-1)^n
q^{kn^2+((k-1)j+k-i)n+j+1}(1-q^{2n})(1-q^{2(n+1)})\cdots(1-q^{2(n+j)})
\nonumber\\
& = \ds\sum\limits_{n\geq 0}(-1)^{n+1}
q^{kn^2+2kn+k+((k-1)j+k-i)(n+1)+j+1}(1-q^{2(n+1)})\cdots(1-q^{2(n+j+1)})
\nonumber\\
& = \ds\sum\limits_{n\geq 0}(-1)^{n+1}
q^{kn^2+((k-1)j+k-i)n + {(2n+j+1)(i-1)}+ (2n+j+2)(k-i+1)}(1-q^{2(n+1)})
\cdots(1-q^{2(n+j+1)}).
\label{first_sum}
\end{align}
Combining \eqref{does_not_change} with \eqref{first_sum}, we arrive at
\begin{align}
& \ds\sum\limits_{n\geq 0}(-1)^n
q^{kn^2+((k-1)j+k-i)n + {(2n+j+1)(i-1)}}(1 - q^{(2n+j+2)(k-i+1)})
(1-q^{2(n+1)})\cdots(1-q^{2(n+j+1)})\nonumber,
\end{align}
which equals
\begin{align}
q^{(j+1)(i-1)}\ds\sum\limits_{n\geq 0}(-1)^n
q^{kn^2+ ((k-1)(j+1) + i-1)n }(1 - q^{(2n+j+2)(k-i+1)})
(1-q^{2(n+1)})\cdots(1-q^{2(n+j+1)})\nonumber.
\end{align}

It is easy to see that the proof above of the equality
\begin{align}
&  
\ds\sum\limits_{n\geq 0}(-1)^n {q^{kn^2+((k-1)j+k-i)n}
(1-q^{(2n+j+1)i})(1-q^{2(n+1)})\cdots(1-q^{2(n+j)})}(1+q^{j+1})\nonumber\\
& \mbox{} - \ds\sum\limits_{n\geq 0}(-1)^n
{q^{kn^2+((k-1)j+k-i)n}(1-q^{2(n+1)})\cdots(1-q^{2(n+j)})
(1+q^{2n+j+1})(1-q^{(2n+j+1)(i-1)})}
\nonumber\\
& = q^{(j+1)(i-1)}\ds\sum\limits_{n\geq 0}(-1)^n
q^{kn^2+ ((k-1)(j+1) + i-1)n }(1 - q^{(2n+j+2)(k-i+1)})
(1-q^{2(n+1)})\cdots(1-q^{2(n+j+1)})
\label{edgematching_prelim}
\end{align}
in fact works even if $i=1$. 
For the $i=1$ case, the second term on the left-hand side of 
\eqref{edgematching_prelim} is equal to 0.
Therefore, dividing through by $(q)_\infty(1+q)\cdots(1+q^{j+1})$ 
precisely gives the edge-matching, i.e.,
$$RHS_{j,k}= RHS_{(j+1),1}.$$

Finally, we use the formulas for $B_{(k-1)(j+1)+i}$ 
along with the edge-matching phenomenon
in order to prove the required formulas for $\B$.
For $i\in\left\{2,\dots,k-1\right\}$, 
$$\B_{(k-1)(j+1)+i}=
\ds\frac{B_{(k-1)(j+1)+i-1} + q^{j+2} B_{(k-1)(j+1)+i+1}}{1+q^{j+2}}.
$$
Therefore,
\begin{align}
(q)_\infty(1+q)&\cdots(1+q^{j+2})\B_{(k-1)(j+1)+i}\nonumber\\
= & \ds\sum\limits_{n\geq 0}(-1)^n
{q^{kn^2+((k-1)(j+1)+i-2)n}(1-q^{(2n+j+2)(k-i+2)})
(1-q^{2(n+1)})\cdots(1-q^{2(n+j+1)})}\nonumber\\
& \mbox{} + \ds\sum\limits_{n\geq 0}(-1)^n
{q^{j+2+kn^2+((k-1)(j+1)+i)n}(1-q^{(2n+j+2)(k-i)})(1-q^{2(n+1)})
\cdots(1-q^{2(n+j+1)})}\nonumber\\
= &\ds\sum\limits_{n\geq 0}(-1)^nq^{kn^2+((k-1)(j+1)+i-2)n}
(1-q^{2(n+1)})\cdots(1-q^{2(n+j+1)})\cdot\nonumber\\
& \mbox{}\cdot
\left(1-q^{(2n+j+2)(k-i+2)}+q^{2n+j+2}-q^{(2n+j+2)(k-i+1)}\right)
\nonumber\\
= &\label{Bt_upto_k-1} 
\ds\sum\limits_{n\geq 0}(-1)^nq^{kn^2+((k-1)(j+1)+i-2)n}
(1-q^{2(n+1)})\cdots(1-q^{2(n+j+1)})\cdot\\
&\cdot(1+q^{2n+j+2})(1-q^{(2n+j+2)(k-i+1)})\nonumber.
\end{align}
For $i=k$,
\[
\B_{(k-1)(j+1)+k}=\ds\frac{B_{(k-1)(j+1)+k-1} }{1+q^{j+2}}.
\]
Therefore,
\begin{align}
(q)_\infty(1+q)&\cdots(1+q^{j+2})\B_{(k-1)(j+1)+k}\nonumber\\
& = \ds\sum\limits_{n\geq 0}(-1)^n
q^{kn^2+((k-1)(j+1)+k-2)n}(1-q^{2(2n+j+2)})(1-q^{2(n+1)})
\cdots(1-q^{2(n+j)})(1-q^{2(n+j+1)})\nonumber,
\end{align}
which equals \eqref{Bt_upto_k-1} with $i=k$.
\end{proof}

\begin{rem}\label{norole}
Note that the factor $(q)_\infty$ played no role whatsoever 
in the proof above, except for the identification of the closed 
forms on the zeroth shelf with the 
products in the Andrews-Bressoud identities.
\end{rem}

\section{The Empirical Hypothesis}
\label{sec:EH}
Using the formula for $B_{(k-1)j+i}$ in Theorem \ref{thm:main},
we see that the $n=0$ term of the sum is
\begin{align*}
\ds \frac{(1-q^{(j+1)(k-i+1)})
(1-q^{2})\cdots(1-q^{2j})}{(q)_\infty(1+q)\cdots(1+q^j)}
& =\ds \frac{(1-q^{(j+1)(k-i+1)})
(1-q)\cdots(1-q^{j})}{(q)_\infty} \\
& =\ds \frac{1-q^{(j+1)(k-i+1)}}{(1-q^{j+1})(1-q^{j+2})\cdots},
\end{align*}
and so
\begin{align*}
B_{(k-1)j+i}
& =\ds \frac{(1-q^{(j+1)(k-i+1)})
(1-q^{2})\cdots(1-q^{2j})}{(q)_\infty(1+q)\cdots(1+q^j)}\\
& \mbox{} + \ds\sum\limits_{n\geq 1}(-1)^n
\frac{q^{kn^2+((k-1)j+i-1)n}(1-q^{(2n+j+1)(k-i+1)})(1-q^{2(n+1)})
\cdots(1-q^{2(n+j)})}{(q)_\infty(1+q)\cdots(1+q^j)}.
\end{align*}
For any $j\ge 0$, $i\in\left\{1,\dots,k\right\}$ and for any $n\ge 1$,
we see that
$$kn^2 +((k-1)j+i-1)n \ge j+2,$$
since $k\ge 2$.
Analyzing the first summand above, we get that
\begin{equation}
B_{(k-1)j+i}=
\begin{cases}
1 + q^{j+1}\gamma_i^{(j+1)}(q) & \text{if }\, 1 \leq i \leq k-1 \\
1 + q^{j+2}\gamma_k^{(j+2)}(q) & \text{if }\, i=k
\end{cases}
\label{eqn:EHforB}
\end{equation}
where 
\begin{equation}\gamma_i^{(j+1)}(q)\in \mathbb{C}[[q]].
\label{eqn:EHforB-gamma}
\end{equation}

Similarly,
\begin{align*}
\B&_{(k-1)j+i}  \\
& = \frac{ 1-q^{(j+1)(k-i+1)} }
{(1-q^{j+1})(1-q^{j+2})\cdots} \\
& \mbox{} + \ds\sum\limits_{n\geq 1}(-1)^n
\frac{ q^{kn^2+((k-1)j+i-2)n}(1-q^{2(n+1)})\cdots
(1-q^{2(n+j)})(1+q^{2n+j+1})(1-q^{(2n+j+1)(k-i+1)}) }
{(q)_\infty(1+q)\cdots(1+q^{j+1})}.
\end{align*}
Hence
\begin{equation}
\B_{(k-1)j+i}=\begin{cases}
1+ q^{j+1}\tilde\gamma_i^{(j+1)}(q) & \text{if }\, 2\leq i\leq k-1\\
1+ q^{j+2}\tilde\gamma_k^{(j+2)}(q) & \text{if }\, i=k
\end{cases}
\label{eqn:EHforBt}
\end{equation}
where 
\begin{equation}\tilde\gamma_i^{(j+1)}(q)\in \mathbb{C}[[q]].
\label{eqn:EHforBt-gamma}
\end{equation}
Equations \eqref{eqn:EHforB}--\eqref{eqn:EHforBt-gamma}
form our Empirical Hypothesis.

\begin{rem}\label{rem:SEH}
Carefully analyzing the various $\gamma$'s and $\tilde\gamma$'s appearing
above, we can make our Empirical Hypothesis stronger, as follows:
\begin{equation}
B_{(k-1)j+i}=\begin{cases}
1 + q^{j+1}+\cdots & \text{if }\, 1 \leq i \leq k-1 \\
1 + q^{j+2}+\cdots & \text{if }\, i=k
\end{cases}
\label{eqn:SEHforB}
\end{equation}
and 
\begin{equation}
\B_{(k-1)j+i}=\begin{cases}
1 + q^{j+1}+\cdots & \text{if }\, 2 \leq i \leq k-1 \\
1 + q^{j+2}+\cdots & \text{if }\, i=k.
\end{cases}
\label{eqn:SEHforBt}
\end{equation}
Equations \eqref{eqn:SEHforB} and \eqref{eqn:SEHforBt} will be 
collectively called the Strong Empirical Hypothesis.
\end{rem}

\begin{rem}
\label{rem:UniquenessofQinfty}
The way we have deduced our Empirical Hypothesis highlights the 
importance of the factor $(q)_\infty$ appearing in the 
denominator of our closed-form expressions --- it is the unique such factor 
that can yield the Empirical Hypothesis.
Recall that by contrast, this factor played no significant role in the proof of 
Theorem \ref{thm:main} (see Remark \ref{norole}).
\end{rem}

\begin{rem}\label{rem:WEH}
As will be clear from the proofs of our main theorems, 
Theorem \ref{thm:infty-main} and 
Theorem \ref{thm:comb-main}, we will only need the information 
related to the series  $B_r$ (and not the series $\B_r$)
from our Empirical Hypotheses.
Moreover, the only form of the Empirical Hypothesis that is 
logically needed to prove the Andrews-Bressoud identities is 
a weaker one, which states that for any positive integer $r$ 
there exists a positive integer
$f(r)$ with
\begin{equation}
B_r \in 1 + q^{f(r)}\mathbb{C}[[q]],
\end{equation}
such that 
\begin{equation}
\lim_{r\rightarrow\infty}f(r) = \infty.
\label{eqn:WEH}
\end{equation}
We will refer to this form as the Weak Empirical Hypothesis.
\end{rem}

\section{Matrix interpretation}\label{sec:matrix}

The aim of this section is to give a matrix formulation of the 
shelf picture obtained so far. Using this formulation, we
present Theorem \ref{thm:infty-main}, which is
one form of our main theorem.

Eliminating the $\B_r$'s from equations 
\eqref{edgematching} -- \eqref{def:higherj}, 
we arrive at the following recursions:
\begin{prop} For $j\ge 0$,
\begin{align}
B_{(k-1)(j+1)+1}&=B_{(k-1)j+k} 
\label{eliminateBt-edgematching}\\
\ds B_{(k-1)(j+1)+2}&= \frac{B_{(k-1)j+k-1}}{1+q^{j+1}}
\label{eliminateBt-2}\\
\ds B_{(k-1)(j+1)+i}&= \frac{B_{(k-1)j+k-i+1}-
B_{(k-1)j+k-i+3}}{q^{(j+1)(i-2)}(1+q^{j+1})}\,\,\,\text{for }\, 
i \in\left\{3,\dots,k\right\}.
\label{eliminateBt-general}
\end{align}
\end{prop}
Note that equation \eqref{eliminateBt-edgematching} gives the edge-matching.
Collectively, equations 
\eqref{eliminateBt-edgematching}--\eqref{eliminateBt-general} 
provide recursions defining the  
higher-shelf $B_r$'s in terms of only lower-shelf $B_r$'s.
Now we reverse this procedure. We can rewrite 
\eqref{eliminateBt-edgematching}--\eqref{eliminateBt-general} as follows:
\begin{align}
B_{(k-1)j+k}&=B_{(k-1)(j+1)+1}\label{recursions-em}\\
\ds {B_{(k-1)j+k-1}} &= B_{(k-1)(j+1)+2}{(1+q^{j+1})}
\label{recursions-2}\\
B_{(k-1)j+i} &= B_{(k-1)(j+1)+k-i+1}{q^{(j+1)(k-i-1)}(1+q^{j+1})}
+B_{(k-1)j+i+2}\,\,\,\text{for }\,i \in\left\{1,\dots,k-2\right\}.
\label{recursions-gen}
\end{align}

Fix an integer
\[
J\ge 0.
\]
This will denote a ``starting'' shelf, which need not be the zeroth shelf.
Using \eqref{recursions-em}--\eqref{recursions-gen}, 
for any $j\ge J$ and $i \in\left\{1,\dots,k\right\}$,  we can write
\begin{align}\label{Bi_h_ijl}
B_{(k-1)J+i} & = 
{}_i^Jh^{(j)}_1 B_{(k-1) j +1} + \cdots + {}_i^Jh^{(j)}_k B_{(k-1) j + k}.
\end{align}
{}From the form of \eqref{recursions-em}--\eqref{recursions-gen}, 
we see that the coefficients ${}_i^Jh^{(j)}_l$ 
are polynomials in $q$ with nonnegative integral coefficients,
and moreover, it is not hard to see that 
\begin{align}
k \equiv 1 \modulo \text{ implies that } 
&{}_i^Jh^{(j)}_l=0 \text{ for } l\not\equiv i\modulo, \nonumber \\
k \equiv 0 \modulo \text{ implies that } 
&{}_i^Jh^{(j)}_l=0 \text{ for } j-J+l\not\equiv i\modulo \label{h_is_0}.
\end{align}

For each $j\ge J$, the coefficients ${}_i^Jh^{(j)}_l$ 
can be assembled to form a $k\times k$ matrix ${}^J{\bf h}^{(j)}$:
\begin{equation}\label{hmatrix}
{}^J{\bf h}^{(j)} = 
\left[\begin{matrix}
{}_1^J{\bf h}^{(j)}\\
\vdots\\
{}_k^J{\bf h}^{(j)}
\end{matrix}
\right]
= 
\left[\begin{matrix}
{}_1^J{h}^{(j)}_1 & \cdots & {}_1^J{h}^{(j)}_k\\
\vdots & \ddots & \vdots\\
{}_k^J{h}^{(j)}_1 & \cdots & {}_k^J{h}^{(j)}_k
\end{matrix}\right].
\end{equation}
The row vectors of ${}^J{\bf h}^{(j)}$ are
\begin{equation}\label{boldh}
{}_i^J{\bf h}^{(j)} = [{}_i^Jh^{(j)}_1,\dots,{}_i^Jh^{(j)}_k].
\end{equation}
For $j=J$, 
\begin{equation}\label{boldh0}
 {}_i^J{\bf h}^{(J)} = [0,\dots,0,1,0,\dots,0],
\end{equation}
where the $1$ is in the $i^{\text{th}}$ position,
and therefore,
\begin{align}\label{Hj=JisIdentity}
{}^J{\bf h}^{(J)} = \text{I}_{k\times k}.
\end{align}
Now, independently of the left subscripts ${}_i$ and superscripts ${}^J$, 
the ${}_i^Jh^{(j)}_l$ satisfy the same recursions with respect to $j$, 
as given in the following proposition:
\begin{prop}\label{h_rec}

\begin{enumerate}
\item If $k$ is odd,
\begin{align}
{}_i^Jh^{(j+1)}_{1} &= ({}_i^Jh^{(j)}_1+{}_i^Jh^{(j)}_3+\cdots + 
{}_i^Jh^{(j)}_{k-2} + {}_i^Jh^{(j)}_k) \nonumber\\
{}_i^Jh^{(j+1)}_{3} &= ({}_i^Jh^{(j)}_1+{}_i^Jh^{(j)}_3+\cdots + 
{}_i^Jh^{(j)}_{k-2})(1+q^{j+1})q^{j+1} \nonumber\\
&\vdots&\label{koddlodd}\\
{}_i^Jh^{(j+1)}_{k-2} &= ({}_i^Jh^{(j)}_1+ {}_i^Jh^{(j)}_3)
(1+q^{j+1})q^{(k-4)(j+1)} \nonumber\\
{}_i^Jh^{(j+1)}_{k} &= {}_i^Jh^{(j)}_1(1+q^{j+1})q^{(k-2)(j+1)} \nonumber
\end{align}
and
\begin{align}
{}_i^Jh^{(j+1)}_{2} &= ({}_i^Jh^{(j)}_2+{}_i^Jh^{(j)}_4+\cdots + 
{}_i^Jh^{(j)}_{k-3}+ {}_i^Jh^{(j)}_{k-1})(1+q^{j+1}) \nonumber\\
{}_i^Jh^{(j+1)}_{4} &= ({}_i^Jh^{(j)}_2+{}_i^Jh^{(j)}_4+\cdots + 
{}_i^Jh^{(j)}_{k-3})(1+q^{j+1})q^{2(j+1)} \nonumber\\
&\vdots& \label{koddleven}\\
{}_i^Jh^{(j+1)}_{k-3} &= ({}_i^Jh^{(j)}_2+ {}_i^Jh^{(j)}_4)
(1+q^{j+1})q^{(k-5)(j+1)} \nonumber\\
{}_i^Jh^{(j+1)}_{k-1} &= {}_i^Jh^{(j)}_2(1+q^{j+1})q^{(k-3)(j+1)}.\nonumber
\end{align}
 
\item If $k$ is even,
\begin{align}
{}_i^Jh^{(j+1)}_{1} &= ({}_i^Jh^{(j)}_2+{}_i^Jh^{(j)}_4+\cdots + 
{}_i^Jh^{(j)}_{k-2} + {}_i^Jh^{(j)}_k) \nonumber\\
{}_i^Jh^{(j+1)}_{3} &= ({}_i^Jh^{(j)}_2+{}_i^Jh^{(j)}_4+\cdots + 
{}_i^Jh^{(j)}_{k-2})(1+q^{j+1})q^{j+1} \nonumber\\
&\vdots& \label{kevenlodd}\\
{}_i^Jh^{(j+1)}_{k-3} &= ({}_i^Jh^{(j)}_2+ {}_i^Jh^{(j)}_4)
(1+q^{j+1})q^{(k-5)(j+1)} \nonumber\\
{}_i^Jh^{(j+1)}_{k-1} &= {}_i^Jh^{(j)}_2(1+q^{j+1})q^{(k-3)(j+1)}.\nonumber
\end{align}
and
\begin{align}
{}_i^Jh^{(j+1)}_{2} &= ({}_i^Jh^{(j)}_1+{}_i^Jh^{(j)}_3+\cdots + 
{}_i^Jh^{(j)}_{k-3}+ {}_i^Jh^{(j)}_{k-1})(1+q^{j+1}) \nonumber\\
{}_i^Jh^{(j+1)}_{4} &= ({}_i^Jh^{(j)}_1+{}_i^Jh^{(j)}_3+\cdots + 
{}_i^Jh^{(j)}_{k-3})(1+q^{j+1})q^{2(j+1)} \nonumber\\
&\vdots& \label{kevenleven}\\
{}_i^Jh^{(j+1)}_{k-2} &= ({}_i^Jh^{(j)}_1+ 
{}_i^Jh^{(j)}_3)(1+q^{j+1})q^{(k-4)(j+1)} \nonumber\\
{}_i^Jh^{(j+1)}_{k} &= {}_i^Jh^{(j)}_1(1+q^{j+1})q^{(k-2)(j+1)}\nonumber.
\end{align}
\end{enumerate}
\end{prop}
\begin{proof}
Immediate from \eqref{recursions-em}--\eqref{Bi_h_ijl}. 
\end{proof}

We put the recursions for the $B_r$'s and $h_l$'s into a succinct matrix form.
For each $j\geq 0$, define the vectors
\begin{equation}\label{Bvect}
{\bf B}_{(j)} = 
\left[\begin{matrix}
B_{(k-1)j + 1}\\
B_{(k-1)j + 2}\\
\vdots\\
B_{(k-1)j + k}
\end{matrix}
\right],
\end{equation}
\begin{equation}\label{Btvect}
{\bf \B}_{(j)} = 
\left[\begin{matrix}
B_{(k-1)j + 2}\\
\vdots\\
B_{(k-1)j + k}
\end{matrix}
\right].
\end{equation}
Set
\begin{align}
{ \mathcal  B}_{(j)} =  
(1+q^j)^{-1}\left[\begin{matrix}
0 & 0 & 0 & \cdots & 0 & 0 & 0 & (1+q^j)\\
0 & 0 & 0 & \cdots & 0 & 0 & 1 & 0\\
0 & 0 & 0 & \cdots & 0 & q^{-j} & 0 & -q^{-j}\\
0 & 0 & 0 & \cdots & q^{-2j} & 0 & q^{-2j} & 0\\
\vdots & \vdots & \vdots & \swarrow & \swarrow & \swarrow & \vdots & \vdots\\
& & & & & & &\\
& & & & & & &\\
q^{-(k-2)j} & 0 & -q^{-(k-2)j} & \cdots & 0 & 0 & 0 & 0
\end{matrix}\right]\label{Bmatrix}.
\end{align}
Then for $j\geq 1$,
\begin{align}
{\bf B}_{(j)} = {\mathcal B}_{(j)}{\bf B}_{(j-1)}.
\end{align}
Recall that we have fixed an integer $J\ge 0$.
For any $j>J$,
\begin{align}
{\bf B}_{(j)} = {\mathcal B}_{(j)}{\mathcal B}_{(j-1)}\cdots 
{\mathcal B}_{(J+1)}{\mathbf B}_{(J)}. 
\end{align}
Recalling \eqref{Bi_h_ijl}, \eqref{hmatrix}--\eqref{Hj=JisIdentity}, 
we immediately have
\begin{align}
{\bf B}_{(J)} = {}^J{\bf h}^{(j)}{\bf B}_{(j)}
\end{align}
for any $j\ge J$.

If $k$ is odd, define a $k\times k$ matrix
\begin{align}
{\mathcal A}_{(j)}=
&(1+q^j)\left[\begin{matrix}
(1+q^j)^{-1} & 0 & q^j & 0             &  & \cdots & 0     & q^{(k-2)j}\\
0            & 1 & 0   & q^{2j} & & \cdots  & q^{(k-3)j} & 0 \\ 
(1+q^j)^{-1} & 0 & q^j & 0      & & \swarrow   & 0     & 0 \\
0            & 1 & 0   & q^{2j} & &  & 0 & 0 \\ 
\vdots & \vdots & \vdots & \vdots &  &  \vdots  & \vdots & \vdots\\
(1+q^j)^{-1} & 0 & q^j & 0      &    & \cdots  & 0 & 0 \\
0            & 1 & 0   & 0      & & \cdots  & 0 & 0 \\ 
(1+q^j)^{-1} & 0 & 0 & 0 & & \cdots & 0 & 0
\end{matrix}\right]\label{Amatrixodd},
\end{align}
and if $k$ is even, define a $k\times k$ matrix
\begin{align}
{\mathcal A}_{(j)}=
&(1+q^j)\left[\begin{matrix}
0 & 1 &  0         & q^{2j} &  & \cdots & 0     & q^{(k-2)j} \\
(1+q^j)^{-1} & 0     & q^j &  0            & & \cdots  & q^{(k-3)j} & 0\\ 
0 & 1 & 0          & q^{2j} & & \swarrow   & 0     & 0 \\
(1+q^j)^{-1} & 0     & q^j & 0             & & \cdots  & 0 & 0 \\ 
\vdots & \vdots & \vdots & \vdots &  &  \vdots  & \vdots & \vdots\\
(1+q^j)^{-1} & 0     & q^j & 0         &    & \cdots  & 0 & 0 \\
0 & 1 & 0          & 0  & & \cdots  & 0 & 0 \\ 
(1+q^j)^{-1} & 0 & 0 & 0 & & \cdots & 0 & 0
\end{matrix}\right]\label{Amatrixeven}.
\end{align}

With this, it is now clear that for any $j\ge J$,
\begin{align}\label{hAequation}
{}^J{\bf h}^{(j+1)} = {}^J{\bf h}^{(j)}{{\mathcal A}_{(j+1)}}
\end{align}
and therefore,
\begin{align}\label{hAequation_j0}
{}^J{\bf h}^{(j)} ={\mathcal A}_{(J+1)}{\mathcal A}_{(J+2)}
\cdots{{\mathcal A}_{(j)}}.
\end{align}
The recursion for the $h_l$'s is the ``inverse'' of the recursion for the $B_r$'s, 
in the precise sense that
\begin{align}
{\mathcal B}_{(j)} = ({\mathcal A}_{(j)})^{-1}.
\end{align}
This gives
\begin{align}
{\mathcal A}_{(j+1)}{\bf B}_{(j+1)} = {\bf B}_{(j)}
\end{align}
for $j\geq 0$ and 
\begin{align}
{\bf B}_{(J)}= {\mathcal A}_{(J+1)}{\mathcal A}_{(J+2)}
\cdots{\mathcal A}_{(j)}{\bf B}_{(j)}
\end{align}
for any $j>J$.

Consider also the $(k-1)\times k$ matrix 
\begin{align}
{ \mathcal  \B}_{(j)}= 
(1+q^{j+2})^{-1}
\left[\begin{matrix}
1 & 0 & q^{j+2} & 0 & \cdots & 0 & 0 & 0\\
0 & 1 & 0 & q^{j+2} & \cdots & 0 & 0 & 0\\
\vdots & \vdots &\vdots &\vdots & \searrow & \vdots &\vdots &\vdots\\
0 & 0 & 0 & 0 & \cdots & 1 & 0 & q^{j+2}\\
0 & 0 & 0 & 0 & \cdots & 0 & 1 & 0
\end{matrix}\right]\label{Btmatrix}.
\end{align}
Then for any $j\geq 0$,
\begin{equation}
{\bf \B}_{(j)} = {\mathcal \B}_{(j)}{\bf B}_{(j)}.
\end{equation}

We say that a sequence of formal power series (or polynomials)
$\left\lbrace P_t(q)=\sum_{s \ge 0}
p_{t,s}q^s\right\rbrace_{t \ge 0}$, with $p_{t,s} \in \mathbb{Z}$,
has a limit if each of the sequences 
$\left\lbrace p_{t,s}\right\rbrace_{t \ge 0}$ has a limit as 
$t\rightarrow\infty$ (that is, the coefficients stablilize).
The limit of $\left\lbrace P_t(q)\right\rbrace$, if it exists, 
is denoted by $\lim_{t\rightarrow\infty}P_t(q)$, and is defined by:
\begin{equation}
\lim_{t\rightarrow\infty}P_t(q) = 
\sum_{s \ge 0} \left(\lim_{t\rightarrow\infty}p_{t,s}\right)q^s.
\end{equation}

Due to the pure powers of $q$ appearing in the right-hand sides of
\eqref{koddlodd}--\eqref{kevenleven},
it is clear that if $l>2$, then
\begin{align}\label{eq:l>1-hJijl=0}
\lim_{j\rightarrow\infty}{}^J_{i}h^{(j)}_l =0. 
\end{align}

For any integer $i$, let
\begin{equation}
\bar{i}\in\{1,2\} \text{ be such that } 
i\equiv\bar{i}\modulo
\label{def:bari}
\end{equation}
and let
\begin{equation}
\bar{i}'\in\{1,2\} \text{ be such that }
i\not\equiv \bar{i}'\modulo.
\label{def:bariprime}
\end{equation}

Now we analyze the cases $l=1$ and $l=2$.
First let $k$ be odd. Let $t$ be a nonnegative integer 
and let $\tau$ be any integer sufficiently larger than $t$. 
Then because of the Empirical Hypothesis, 
\eqref{Bi_h_ijl}, \eqref{h_is_0} and \eqref{eq:l>1-hJijl=0},
we conclude that the coefficient of $q^t$ in $B_{(k-1)j+i}$ 
is the same as the coefficient of $q^t$ in ${}^J_ih^{(\tau)}_{\bar{i}}$. 
This shows that the coefficient of any power of $q$ in 
${}^J_ih^{(j)}_{\bar{i}}$ stabilizes as $j\rightarrow\infty$. 
Hence
\begin{align}
\lim_{j\rightarrow\infty}{}^J_ih^{(j)}_{\bar{i}}(q)
\end{align}
exists, which we denote by
\begin{align}
{}^J_ih^{(\infty)}(q)\label{hinftykodd}.
\end{align}
Now let $k$ be even. 
Consider the sequence of polynomials
\begin{align}
{}^J_ih^{(J)}_{\bar{i}}(q),\, {}^{J}_ih^{(J+1)}_{\bar{i}'}(q),\, 
{}^{J}_ih^{(J+2)}_{\bar{i}}(q),\, {}^{J}_ih^{(J+3)}_{\bar{i}'}(q), \dots.
\end{align}
The general term of this sequence is ${}^J_ih^{(j)}_{\overline{j-J+i}}(q)$ 
for $j\geq J$.
Similarly as above, because of the Empirical Hypothesis, 
\eqref{Bi_h_ijl}, \eqref{h_is_0} and \eqref{eq:l>1-hJijl=0},
the coefficient of $q^t$ in $B_{(k-1)j+i}$ is the same as the coefficient
of $q^t$ in ${}^J_ih^{(j')}_{\overline{j'-J+i}}$ 
for any $j'$ sufficiently larger than $t$. 
This shows that the coefficient of any power of $q$ in 
${}^J_ih^{(j)}_{\overline{j-J+i}}$ stabilizes as $j\rightarrow\infty$. Hence
\begin{align}
\lim_{j\rightarrow\infty}{}^J_ih^{(j)}_{\overline{j-J+i}}(q)
\end{align}
exists, which we again denote by
\begin{align}
{}^J_ih^{(\infty)}(q)\label{hinftykeven}.
\end{align}
It is worth noting that when $k$ is odd, $\left\lbrace  
{}^J_{i}h^{(j)}_{\bar{i}'}\right\rbrace_{j \ge J}$ is a sequence of
zero polynomials, and when $k$ is even, 
$\left\lbrace  {}^J_{i}h^{(j)}_{\overline{j-J+i}'}
\right\rbrace_{j \ge J}$ is again a 
sequence of zero polynomials.
With this discussion, we have in fact proved one form of our main theorem:
\begin{thm}\label{thm:infty-main}
For any $J\ge 0$ and $i\in\left\{1,\dots,k\right\}$,
\begin{align}\label{eq:infty-main}
B_{(k-1)J+i} =  {}^J_ih^{(\infty)}(q).
\end{align}
\end{thm}
\begin{proof}
Let $k$ be odd. It follows from the Empirical Hypothesis, 
\eqref{Bi_h_ijl}, \eqref{h_is_0}, \eqref{eq:l>1-hJijl=0} 
and the preceeding discussion that
\begin{align*}
B_{(k-1)J+i}&
= \left(\lim_{j\rightarrow\infty} {}^J_ih^{(j)}_{\bar{i}}\right)
\left(\lim_{j\rightarrow\infty}B_{(k-1)j+\bar{i}}\right)
+\left(\lim_{j\rightarrow\infty} {}^J_ih^{(j)}_{\bar{i}'}\right)
\left(\lim_{j\rightarrow\infty}B_{(k-1)j+\bar{i}'}\right)\\
& \quad+\left(\lim_{j\rightarrow\infty} {}^J_ih^{(j)}_{3}\right)
\left(\lim_{j\rightarrow\infty}B_{(k-1)j+3}\right)
+\cdots+\left(\lim_{j\rightarrow\infty} {}^J_ih^{(j)}_{k}\right)
\left(\lim_{j\rightarrow\infty}B_{(k-1)j+k}\right)\\
&= \left(\lim_{j\rightarrow\infty} {}^J_ih^{(j)}_{\bar{i}}\right)
\cdot 1+0+\cdots+0\\
&=  {}^J_ih^{(\infty)}.
\end{align*}
Similarly, if $k$ is even,
\begin{align*}
B_{(k-1)J+i}&
= \left(\lim_{j\rightarrow\infty} {}^J_ih^{(j)}_{\overline{j-J+i}}\right)
\left(\lim_{j\rightarrow\infty}B_{(k-1)j+\overline{j-J+i}}\right)
+\left(\lim_{j\rightarrow\infty} {}^J_ih^{(j)}_{\overline{j-J+i}'}\right)
\left(\lim_{j\rightarrow\infty}B_{(k-1)j+\overline{j-J+i}'}\right)\\
& \quad+\left(\lim_{j\rightarrow\infty} {}^J_ih^{(j)}_{3}\right)
\left(\lim_{j\rightarrow\infty}B_{(k-1)j+3}\right)+
\cdots+\left(\lim_{j\rightarrow\infty} {}^J_ih^{(j)}_{k}\right)
\left(\lim_{j\rightarrow\infty}B_{(k-1)j+k}\right)\\
&=\left(\lim_{j\rightarrow\infty} {}^J_ih^{(j)}_{\overline{j-J+i}}\right)
\cdot 1+0+\cdots+0\\
&={}^J_ih^{(\infty)}.
\end{align*}
For both of the cases, in the first equalities, only the first
two terms depend on the parity of $i$. The rest of the terms, corresponding
to the indices $3,\dots,k$, appear in their natural order. 
\end{proof}

\section{Combinatorial interpretations}\label{sec:comb}

In this section, we will give combinatorial interpretations of 
various polynomials and formal formal power series discussed thus far. 
In particular, we will prove the Andrews-Bressoud identities in this section.

For $j\geq J+1$ and for $i,l\in \left\{1,\dots,k\right\}$, 
${}_i^Jh^{(j)}_l$ is a polynomial with nonnegative integral coefficients.
Write
\[ {}_i^Jh^{(j)}_l = \sum_{n\geq 0} {}_i^Jh^{(j)}_l(n) q^n. \]

\begin{thm}\label{hcomb}
If $k$ is odd, then ${}_i^Jh^{(j)}_l=0$ for 
$l\not\equiv i \modulo$, and if $k$ is even,
then ${}_i^Jh^{(j)}_l=0$ for $j-J+l\not\equiv i \modulo$.
Otherwise, ${}_i^Jh^{(j)}_l(n)$ equals the number of partitions
$\pi = (\pi_1, \pi_2, \dots, \pi_{\ell(\pi)})$ of $n$ such that:
\begin{enumerate}
\item $\pi_{t} - \pi_{t+k-1} \geq 2$ for all positive integers 
$t$ for which $t+k-1\leq \ell(\pi)$,
\item $\pi_{1} \leq j$,
\item $m_{j}({\pi}) \in \{l-2,l-1\}\cap \mathbb{N}$,
\item $\pi_{\ell(\pi)}>J$,
\item $m_{J+1}(\pi) \leq k-i$,
\item $\pi_t - \pi_{t+k-2} \leq 1$ only if 
$\pi_t + \pi_{t+1} + \cdots + \pi_{t+k-2} 
\equiv (k-1)j+l-k \pmod{2}$ for all positive integers 
$t$ for which $t+k-2\leq \ell(\pi)$.
\end{enumerate}
\end{thm}

\begin{proof}
{}From Proposition \ref{h_rec}  it is clear that:
\begin{itemize}
\item[(i)]  
If $k$ is odd and if ${}_i^Jh^{(j)}_{l'}=0$
for all $l'$ with $l'\not\equiv i\modulo$, then 
${}_i^Jh^{(j+1)}_l=0$ for all $l$ with 
$l\not\equiv i\modulo$. 
\item[(ii)]
If $k$ is even and if ${}_i^Jh^{(j)}_{l'}=0$ 
for all $l'$ with 
$j-J+l'\not\equiv i \modulo$, then 
${}_i^Jh^{(j+1)}_l=0$ for all $l$ with 
$j-J+l\equiv i \modulo$, i.e.,
$(j+1)-J+l\not \equiv i \modulo$.
\end{itemize}
For the remaining cases, 
that is, when $k$ is odd and $l\equiv i \modulo$, or when
$k$ is even and $j-J+l\equiv i \modulo$,
it is now enough to show that the generating functions 
for the partitions described above have the
same recursions and initial conditions as ${}_i^Jh^{(j)}_l$. 

We say that a partition $\pi$ is of type $(n\,;\,k-1,j,S,J,k-i,p)$ 
if it partitions $n$ and satisfies conditions (1) -- (6) 
above with the set $\{l-2,l-1\}$ in (3) replaced by the set $S$
and the congruence in (5) replaced by $\equiv p\, (\text{mod } 2)$. 
Also let $|(\cdots)|$ be the number of partitions of type $(\cdots)$.

%%%%%%%%%%%%%%%%
First, suppose that $l \in \left\{1,\dots,k\right\}$
and that $\pi$ is a partition of type 
$(n\,;\,k-1,j+1,\{l-2,l-1\}\cap\mathbb{N}, J, k-i, (k-1)(j+1)+l-k)$.
We claim that for such a $\pi$, $m_{j}(\pi)\leq k-l$.
Indeed, it is clear that condition (1) forces 
$m_{j+1}(\pi)+m_{j}(\pi)\leq k-1$. 
Now, in the case that $m_{j+1}(\pi)=l-1$ and $m_{j}(\pi) = k-l$,
$(l-1)(j+1)+(k-l)j =  (k-1)(j+1)+l-k$ (and hence, congruent mod 2).
But when $l>1$, $m_{j+1}(\pi)=l-2$ and 
$m_{j}(\pi)=k-l+1$, $(l-2)(j+1)+(k-l+1)j= (k-1)(j+1)+l-k-1$
making $(l-2)(j+1)+(k-l+1)j$ and  $(k-1)(j+1)+l-k$ noncongruent mod 2,
thereby forcing $m_{j}(\pi)<k-l+1$.
With this analysis, it is clear that for such a $\pi$,
$m_{j}(\pi)\leq k-l$, and in fact, 
$m_{j}(\pi)$ can assume any value from $\left\{0,\dots,k-l\right\}$.
%%%%%%%%%%%%%%%%

Now we formulate several cases. 
The proofs in all the cases follow the same pattern, but for
ease of comprehension, we spell out each case in detail.
\begin{enumerate}
\item[(a)]  $k$ is odd and $l>1$:

\noindent 
In this case, the parity condition (5) is independent of 
$j$ since $k$ is odd, i.e., $$(k-1)(j+1)+l-k\equiv (k-1)j+l-k\modulo.$$ 
Also, the disjoint union of the sets $\{t-2,t-1\}\cap\mathbb{N}$ 
as $t$ varies over ${1\leq t \leq k-l+1}$ with the restriction 
that ${t \equiv l \modulo}$ is exactly $\{0,\dots,k-l\}$.
Therefore,
\begin{align*}
&|(n\,;\,k-1,j+1,\{l-2,l-1\},J,k-i,(k-1)(j+1)+l-k)|\nonumber\\
& = |(n-(j+1)(l-1)\,;\,k-1,j,\{0,\dots,k-l\},J,k-i,(k-1)(j+1)+l-k)|\nonumber\\
& \quad\mbox{} + |(n-(j+1)(l-2)\,;\,k-1,j,\{0,\dots,k-l\},J,k-i,(k-1)(j+1)+l-k)|
\nonumber\\
& = |(n-(j+1)(l-1)\,;\,k-1,j,\{0,\dots,k-l\},J,k-i,(k-1)j+l-k)|\nonumber\\
& \quad\mbox{} + |(n-(j+1)(l-2)\,;\,k-1,j,\{0,\dots,k-l\},J,k-i,(k-1)j+l-k)|
\nonumber\\
& = \ds\sum_{\genfrac{}{}{0pt}{}{1\leq t \leq k-l+1}{t\, \equiv\, l \modulo}} 
\left|(n-(j+1)(l-1)\,;\,k-1,j,\{t-2,t-1\}\cap\mathbb{N},J,k-i,(k-1)j+t-k)\right|\\
& \quad\mbox{} + \ds\sum_{\genfrac{}{}{0pt}{}{1\leq t \leq k-l+1}
{t \,\equiv\, l \modulo}} 
|(n-(j+1)(l-2)\,;\,k-1,j,\{t-2,t-1\}\cap\mathbb{N},J,k-i,(k-1)j+t-k)|.
\end{align*}

\item[(b)] $k$ is odd and $l$ is 1: 

\noindent  
Note again that, 
$(k-1)(j+1)+1-k\equiv (k-1)j+1-k\modulo$
since $k$ is odd and the disjoint union of the sets 
$\{t-2,t-1\}\cap\mathbb{N}$ as $t$ varies over 
${1\leq t \leq k}$ with ${t \equiv 1 \modulo}$ is $\{0,\dots,k-1\}$.
Therefore,
\begin{align*}
& |(n,k-1,j+1,\{0\},J,k-i,(k-1)(j+1)+1-k)|\nonumber\\
& = |(n,k-1,j,\{0,\dots,k-1\},J,k-i,(k-1)j+1-k)|\nonumber\\
& = \ds\sum_{\genfrac{}{}{0pt}{}{1\leq t \leq k}{t \equiv 1 \modulo}} 
\left|(n\,;\,k-1,j,\{t-2,t-1\}\cap\mathbb{N},J,k-i,(k-1)j+t-k)\right|
\nonumber.
\end{align*}

\item[(c)]  $k$ is even and $l>1$:

\noindent 
In this case, 
$(k-1)(j+1)+l-k\not\equiv (k-1)j+l-k\modulo,$ or
$(k-1)(j+1)+l-k\equiv (k-1)j+(l+1)-k\modulo.$ 
Also, the disjoint union of the sets $\{t-2,t-1\}\cap\mathbb{N}$ as 
$t$ varies over 
${1\leq t \leq k-l+1}$ with ${t \not\equiv l \modulo}$ is $\{0,\dots,k-l\}$.
Therefore,
\noindent\begin{align*}
&|(n\,;\,k-1,j+1,\{l-2,l-1\},J,k-i,(k-1)(j+1)+l-k)|\nonumber\\
& = |(n-(j+1)(l-1)\,;\,k-1,j,\{0,\dots,k-l\},J,k-i,(k-1)(j+1)+l-k)|
\nonumber\\
& \quad\mbox{} 
+|(n-(j+1)(l-2)\,;\,k-1,j,\{0,\dots,k-l\},J,k-i,(k-1)(j+1)+l-k)|
\nonumber\\
& =|(n-(j+1)(l-1)\,;\,k-1,j,\{0,\dots,k-l\},J,k-i,(k-1)j+(l+1)-k)|
\nonumber\\
& \quad\mbox{} 
+|(n-(j+1)(l-2)\,;\,k-1,j,\{0,\dots,k-l\},J,k-i,(k-1)j+(l+1)-k)|
\nonumber\\
& = \ds\sum_{\genfrac{}{}{0pt}{}{1\leq t \leq k-l+1}{t \not\equiv l \modulo}} 
\left|(n-(j+1)(l-1)\,;\,k-1,j,\{t-2,t-1\}\cap
\mathbb{N},J,k-i,(k-1)j+t-k)\right|\\
& \quad\mbox{} +\ds\sum_{\genfrac{}{}{0pt}{}{1\leq t \leq k-l+1}
{t \not\equiv l \modulo}} 
|(n-(j+1)(l-2)\,;\,k-1,j,\{t-2,t-1\}\cap\mathbb{N},J,k-i,(k-1)j+t-k)|\nonumber.
\end{align*}

\item[(d)] $k$ is even and $l$ is 1: 

\noindent Similarly as in the previous case, 
$(k-1)(j+1)+1-k\equiv (k-1)j-k\modulo$
since $k$ is even and the disjoint union of the sets 
$\{t-2,t-1\}\cap\mathbb{N}$ as $t$ varies over 
${1\leq t \leq k}$ with ${t \not\equiv 1 \modulo}$ is $\{0,\dots,k-1\}$. 
Therefore,
\begin{align*}
& |(n,k-1,(j+1),\{0\},J,k-i,(k-1)(j+1)+1-k)|\nonumber\\
& = |(n,k-1,j,\{0,\dots,k-1\},J,k-i,(k-1)j-k)|\nonumber\\
& = \ds\sum_{\genfrac{}{}{0pt}{}{1\leq t \leq k}{t \not\equiv 1 \modulo}} 
\left|(n\,;\,k-1,j,\{t-2,t-1\}\cap\mathbb{N},J,k-i,(k-1)j+t-k)\right|.
\end{align*}
\end{enumerate}

The discussion above leads us to the conclusion that the generating functions 
for the partitions satisfying conditions (1) -- (6) above have the same recursions
as the polynomials ${}^J_ih^{(j)}_l$.

It is easy to list the partitions satisfying the conditions 
(1) -- (6) when $j=J+1$, and similarly, using Proposition \ref{h_rec} and equation 
\eqref{Hj=JisIdentity}, it is easy to find the polynomials ${}_i^Jh^{(J+1)}_l$.
The following cases show that the initial values match those of the generating functions:
\begin{enumerate}
\item $k$ odd, $i$ odd: 
$$ {}_i^J\mathbf{h}^{(J+1)} = 
[ 1, 0, q^{J+1}+q^{2(J+1)},0,q^{3(J+1)}+q^{4(J+1)},0,
\cdots,0,q^{(k-i-1)(J+1)}+q^{(k-i)(J+1)},0,\cdots,0]. $$
 \item $k$ odd, $i$ even: $$ {}_i^J\mathbf{h}^{(J+1)} = 
[ 0, 1+q^{J+1}, 0,q^{2(J+1)}+q^{3(J+1)},0,
\cdots,0,q^{(k-i-1)(J+1)}+q^{(k-i)(J+1)},0,\cdots,0]. $$
 \item $k$ even, $i$ odd:  $$ {}_i^J\mathbf{h}^{(J+1)} = 
[ 0, 1+q^{J+1}, 0,q^{2(J+1)}+q^{3(J+1)},0,
\cdots,0,q^{(k-i-1)(J+1)}+q^{(k-i)(J+1)},0,\cdots,0]. $$
 \item $k$ even, $i$ even:$$ {}_i^J\mathbf{h}^{(J+1)} = 
[ 1, 0, q^{J+1}+q^{2(J+1)},0,q^{3(J+1)}+q^{4(J+1)},0,
\cdots,0,q^{(k-i-1)(J+1)}+q^{(k-i)(J+1)},0,\cdots,0].$$
\end{enumerate}

\end{proof}

With this, we are now ready to give combinatorial interpretations
of all of the official series and the ghosts.
\begin{thm}\label{thm:comb-main-J}
For $r = (k-1)J + i$, with $i\in\left\{1,\dots,k\right\}$,
\[ B_r = \sum_{n\geq 0} b_{k,r}(n) q^n, \]
where $b_{k,r}(n)$ denotes the number of partitions 
$\pi = (\pi_1, \pi_2, \dots, \pi_{\ell(\pi)})$ of $n$ such that:
\begin{itemize}
\item $\pi_{t} - \pi_{t+k-1} \geq 2$,
\item $\pi_{\ell(\pi)} \geq J+1$,
\item $m_{J+1}(\pi) \leq k-i$,
\item $\pi_t - \pi_{t+k-2} \leq 1$ only if 
$\pi_t+ \pi_{t+1}+ \cdots+ \pi_{t+k-2}\equiv 
r+k \equiv (k-1)J+i+k \pmod{2}$.
\end{itemize}
\end{thm}
\begin{proof}
Let $k$ be odd. Recall the notation \eqref{def:bari}.
Using Theorem \ref{hcomb},
we see that ${}^J_ih^{(j)}_{\overline{i}}(q)$ is the generating function for
partitions satisfying:
\begin{enumerate}
\item $\pi_{t} - \pi_{t+k-1} \geq 2$ for all positive integers 
$t$ for which $t+k-1\leq \ell(\pi)$,
\item $\pi_{1} \leq j$,
\item $m_{j}({\pi}) \in \{\overline{i}-2,\overline{i}-1\}\cap \mathbb{N}$,
\item $\pi_{\ell(\pi)}>J$,
\item $m_{J+1}(\pi) \leq k-i$,
\item $\pi_t - \pi_{t+k-2} \leq 1$ only if $\pi_t + \pi_{t+1}+ \cdots +\pi_{t+k-2} 
\equiv (k-1)j+\bar{i}-k \pmod{2}$ for all positive integers $t$ 
for which $t+k-2\leq \ell(\pi)$.
\end{enumerate}
Now, since $k-1$ is even, and since $\bar{i}\equiv i\modulo$, we see that 
$(k-1)j+\bar{i}-k \equiv (k-1)J+i-k\modulo.$
Also, regardless of the value of $\overline{i}$, ${}^J_ih^{(j)}_{\overline{i}}(q)$ 
equals, up to the coefficient of $q^j$, with the generating function 
for partitions satisfying conditions (1) -- (6) above, 
but with condition (3) replaced by: $$m_{j}({\pi}) =0. $$
We conclude that $B_r$, which equals 
${}^J_ih^{(\infty)}$ by Theorem \ref{thm:infty-main}, is the generating 
function for partitions satisfying the required conditions.

Now let $k$ be even. Using Theorem \ref{hcomb}, we see that 
${}^J_ih^{(j)}_{\overline{j-J+i}}(q)$ is the generating function for
partitions satisfying:
\begin{enumerate}
\item $\pi_{t} - \pi_{t+k-1} \geq 2$ for all positive integers 
$t$ for which $t+k-1\leq \ell(\pi)$,
\item $\pi_{1} \leq j$,
\item $m_{j}({\pi}) \in \{\overline{j-J+i}-2,\overline{j-J+i}-1\}\cap \mathbb{N}$,
\item $\pi_{\ell(\pi)}>J$,
\item $m_{J+1}(\pi) \leq k-i$,
\item $\pi_t - \pi_{t+k-2} \leq 1$ only if 
$\pi_t + \pi_{t+1} + \cdots + \pi_{t+k-2} 
\equiv (k-1)j+\overline{j-J+i}-k \pmod{2}$ for all positive integers 
$t$ for which $t+k-2\leq \ell(\pi)$.
\end{enumerate}
Since $k$ is even, 
$(k-1)j+\overline{j-J+i}-k \equiv -J + i \equiv (k-1)J+i-k \modulo.$
Just like before, 
regardless of the value of $\overline{j-J+i}$, ${}^J_ih^{(j)}_{\overline{j-J+i}}(q)$ 
equals, up to the coefficient of $q^j$, with the generating 
function for partitions satisfying conditions (1) -- (6)
above, but with condition (3) replaced by: $$m_{j}({\pi}) = 0. $$
We conclude that $B_r$, which equals ${}^J_ih^{(\infty)}$ by 
Theorem \ref{thm:infty-main}, is the generating function for
partitions satisfying the required conditions.
\end{proof}

Specializing $J=0$ in Theorem \ref{thm:comb-main-J}, 
we arrive at the Andrews-Bressoud identities:

\begin{thm}\label{thm:comb-main}
For $i\in\left\{1,\dots,k\right\}$,
\[ B_i = \sum_{n\geq 0} b_{k,i}(n) q^n, \]
where $b_{k,i}(n)$ denotes the number of partitions 
$\pi = (\pi_1, \pi_2, \dots, \pi_{\ell(\pi)})$ of $n$ such that:
\begin{enumerate}
\item $\pi_{t} - \pi_{t+k-1} \geq 2$,
\item $m_{1}(\pi) \leq k-i$,
\item $\pi_t - \pi_{t+k-2} \leq 1$ only if 
$\pi_t + \pi_{t+1} + \cdots + \pi_{t+k-2} \equiv i+k \pmod{2}$.
\end{enumerate}
\end{thm}

We now present the combinatorial interpretation of the ghosts.

\begin{thm} \label{thm:comb-main-ghosts}
For $r = (k-1)J + i$, with $i\in\left\{2,\dots,k\right\}$,
\[ \B_r = \sum_{n\geq 0} \tilde{b}_{k,r}(n) q^n, \]
where $\tilde{b}_{k,r}(n)$ denotes the number of partitions 
$\pi = (\pi_1, \pi_2, \dots, \pi_{\ell(\pi)})$ of $n$ such that: 
\begin{enumerate}
\item $\pi_{t} - \pi_{t+k-1} \geq 2$,
\item $\pi_{\ell(\pi)} \geq J+1$,
\item $m_{J+1}(\pi) \leq k-i$,
\item $\pi_t- \pi_{t+k-2} \leq 1$ only if 
$\pi_t+ \pi_{t+1}+ \cdots +\pi_{t+k-2} \equiv 
r+k+1 \equiv (k-1)J+i+k+1 \pmod{2}$.
\end{enumerate}
\end{thm}
\begin{proof}
Using the left- and right-hand sides of
\eqref{def:higherjedge} and \eqref{def:higherj} 
and making appropriate index changes, we find that 
for $J\ge 0$ and $i\in\left\{2,\dots,k-1\right\}$,
\begin{align}
\B_{(k-1)J+k}&=B_{(k-1)(J+1)+2},\label{eqn:ghostformulasedge}\\
\B_{(k-1)J+i}&= q^{(J+1)(k-i)}B_{(k-1)(J+1)+k-i+2}+B_{(k-1)J+i+1}.
\label{eqn:ghostformulas}
\end{align}

Let us first consider the case of $\B_{(k-1)J+k}$. 
Let $\pi$ be a partition counted in $\tilde{b}_{k,(k-1)J+k}(n)$.
Then, conditions (2) and (3) for $\pi$ imply that $\pi_{\ell(\pi)} \geq J+2$.
Moreover, if $m_{J+2}(\pi) = k-1$, then $(k-1)(J+2)\not\equiv (k-1)J+k+k+1 \pmod{2}$,
and therefore condition (4) gets violated. Hence, $m_{J+2}(\pi) \leq k-2$,
and $\pi$ satisfies exactly the conditions satisfied by the partitions counted in 
$b_{k,(k-1)(J+1)+2}(n)$.

Now we consider the case of $\B_{(k-1)J+i}$
for $i\in\left\{2,\dots,k-1\right\}$.
Let $\pi$ be a partition counted in $\tilde{b}_{k,(k-1)J+i}(n)$.
If $m_{J+1}(\pi)=k-i$, then $m_{J+2}(\pi)<i-1$, because
$m_{J+2}(\pi)=i-1$ implies that $(J+1)(k-i)+(J+2)(i-1)\equiv (k-1)J+i+k\modulo$, 
thereby violating the condition (4) for $\pi$.
Moreover, both $B_{(k-1)(J+1)+k-i+2}$ and $B_{(k-1)J+i+1}$ 
have the same ``parity'' condition as that of $\pi$, because
$(k-1)(J+1)+k-i+2+k\equiv  (k-1)J+k+i+1 \modulo.$
Therefore, if $m_J(\pi)=k-i$, it gets counted in
$q^{(J+1)(k-i)}B_{(k-1)(J+1)+k-i+2}$,
else if $m_J(\pi)<k-i$, it is counted in  $B_{(k-1)J+i+1}$.
\end{proof}

\begin{rem}\label{rem:1stghostmeaning}
Recall from Remark \ref{rem:1stghost} that it was not necessary to define $\B_1$.
We may now define it as
\begin{equation}
\B_1 = B_2\label{def:1stghost},
\end{equation}
in the light of equation \eqref{eqn:ghostformulasedge}
(which corresponds to the $i=k$ case of Theorem \ref{thm:comb-main-ghosts}),
with $J=-1$.
With this, the statement of Theorem \ref{thm:comb-main-ghosts}
can be seen to hold for $\B_1$ as well.
\end{rem}

\begin{rem}\label{rem:robinson}(Cf. Remark 2.1 in \cite{LZ}.)
As discussed in \cite{AB}, \cite{R} and \cite{A4}, we can give
an alternate, shorter proof of Theorems \ref{thm:comb-main} 
and \ref{thm:comb-main-ghosts} using
only the Empirical Hypothesis and without using the combinatorial
interpretation of the ${}^j_ih^{(j)}_l$ polynomials as given in
Proposition \ref{hcomb}.
Let $D_1(q),D_2(q),\dots$ and $\tilde{D}_1(q),\tilde{D}_2(q),\dots$
be sequences of formal power series in $q$ 
with constant term $1$ satisfying the recursions \eqref{def:higherjedge} and
\eqref{def:higherj} (with $D_p(q)$'s in place of the $B_p(q)$'s
and $\tilde{D}_p(q)$'s in place of the $\B_p(q)$'s), and suppose that
the Empirical Hypothesis holds for $D_1(q),D_2(q),\dots$
and $\tilde{D}_1(q),\tilde{D}_2(q),\dots$.
Then, in fact, $D_1(q),D_2(q),\dots,D_k(q)$ and 
$\tilde{D}_1(q),\tilde{D}_2(q),\dots$\
are uniquely determined by these recursions and the Empirical Hypothesis. 
Now, on the one hand, it is easy to see that the official sum sides and the
ghost sum sides, i.e., the combinatorial generating functions appearing 
in the statements of Theorems \ref{thm:comb-main-J} and
\ref{thm:comb-main-ghosts}, respectively,
trivially satisfy the recursions just mentioned and the Empirical
Hypothesis as well.
On the other hand, the sequences $B_1(q), B_2(q), \dots$
and $\tilde{B}_1(q), \tilde{B}_2(q), \dots$ satisfy the recursions by definition and
we have proved in Section \ref{sec:EH} that these sequences also satisfy the 
Empirical Hypothesis.
Therefore, by the aforementioned uniqueness, it must be that 
$B_r(q)$ and $\B_r(q)$ equal the combinatorial generating functions appearing in the 
statements of Theorems \ref{thm:comb-main-J} and
\ref{thm:comb-main-ghosts} respectively.
\end{rem}

\section{An $(x,q)$-dictionary}\label{sec:(x,q)dict}
In this section, we obtain a dictionary between our official expressions 
on the various shelves with relevantly specialized $\Jt$ expressions 
appearing in \cite{CoLoMa}. 
Note that this section is concerned only with the official series.
First we let $a\mapsto 0$ in \cite{CoLoMa}.
Thus we get:  For $i \in\left\{1,\dots,k\right\}$,
\begin{equation}
\label{coloma:H} 
\Ht_{k,i}(x,q) = \ds\sum\limits_{n\geq 0}
(-1)^n\frac{q^{kn^2+n-in}x^{(k-1)n}(1-x^iq^{2ni})(-x)_n}
{(-q)_n(q)_n(xq^n)_\infty},
\end{equation}
{}from (1.6) in~\cite{CoLoMa}.
Therefore, using (1.5) in~\cite{CoLoMa},
\begin{align}
\Jt_{k,i}(x,q)&=\Ht_{k,i}(xq,q) = \ds\sum\limits_{n\geq 0}
(-1)^n\frac{q^{kn^2+(k-i)n}x^{(k-1)n}(1-x^iq^{(2n+1)i})(-xq)_n}
{(-q)_n(q)_n(xq^{n+1})_\infty} \label{coloma:Jtilde}\\
\Jt_{k,i}(q^j,q)
&= \ds\sum\limits_{n\geq 0}
(-1)^n\frac{q^{kn^2+((k-1)j+k-i)n}(1-q^{(2n+j+1)i})(-q^{j+1})_n}
{(-q)_n(q)_n(q^{n+j+1})_\infty}.
\label{coloma:Jspecialized_primary} 
\end{align}
Now we fill the ``gap'' in the denominator
by multiplying and dividing the $n$-th summand
by $(1-q^{n+1})\cdots(1-q^{n+j})$:
\begin{align}\ds
\frac{(-q^{j+1})_n }{(-q)_n(q)_n(q^{n+j+1})_\infty}
=&\ds\frac{(1+q^{j+1}) \cdots(1+q^{j+n})(1-q^{n+1})\cdots(1-q^{n+j})}
{(1+q)\cdots(1+q^n)(q)_\infty}.
\label{qsubfactors} 
\end{align}
If $n>j$, we cancel the extra terms in the 
denominator in the right-hand side, arriving at 
\begin{align*}\ds
\frac{(-q^{j+1})_n }{(-q)_n(q)_n(q^{n+j+1})_\infty}
= &\ds\frac{(1+q^{n+1}) \cdots(1+q^{j+n})(1-q^{n+1})
\cdots(1-q^{n+j}) }{(1+q)\cdots(1+q^j)(q)_\infty}\\
= &\ds\frac{(1-q^{2(n+1)})\cdots(1-q^{2(n+j)}) }{(1+q)\cdots(1+q^j)(q)_\infty}.
\end{align*}
Otherwise, if $n\leq j$, we multiply and divide by
$(1+q^{n+1})\cdots(1+q^j)$ to arrive at
\begin{align*}\ds
\frac{(-q^{j+1})_n }{(-q)_n(q)_n(q^{n+j+1})_\infty}
&=\ds\frac{(1+q^{n+1})\cdots(1+q^j)(1+q^{j+1}) \cdots(1+q^{j+n})(1-q^{n+1})
\cdots(1-q^{n+j}) }{(1+q)\cdots(1+q^n)(1+q^{n+1})\cdots(1+q^j)(q)_\infty}\\
&=\ds\frac{(1-q^{2(n+1)})\cdots(1-q^{2(n+j)}) }{(1+q)\cdots(1+q^j)(q)_\infty}.
\end{align*}
Putting these cases together and making the change $i\mapsto k-i+1$, 
we conclude that
\begin{equation}
\label{coloma:Jspecialized_final} 
\Jt_{k,k-i+1}(q^j,q)
= \ds\sum\limits_{n\geq 0}
(-1)^n\frac{q^{kn^2+((k-1)j+i-1)n}(1-q^{(2n+j+1)(k-i+1)})
 (1-q^{2(n+1)})\cdots(1-q^{2(n+j)})}{(q)_\infty(1+q)\cdots(1+q^j)}.\nonumber
\end{equation}
Theorem \ref{thm:main} now yields the following ``dictionary'':
\begin{equation}
\label{dictionary}
B_{(k-1)j+i}=\Jt_{k,k-i+1}(q^j,q)
\end{equation}
for $j\ge0$ and $i\in\left\{1,\dots,k\right\}$.

\section{An {$(x,q)$}-expression governing the ghosts}
\label{sec:(x,q)forghosts}

In this section, we derive a (new) $(x,q)$-expression governing 
the ghost series.
Specializing $x$ to successively higher powers of $q$ in this expression
gives us back the closed-form expressions for the ghosts already obtained in 
Theorem \ref{thm:main}.

For $i\in\left\{2,\dots,k-1\right\}$, let
\begin{equation}
\label{ghost_Jxq_upto_2}
\tilde\Jt_{k,i}(x,q) = \ds\frac{\Jt_{k,i+1}(x,q)+xq\Jt_{k,i-1}(x,q)}{1+xq}, 
\end{equation}
and for $i=1$,
let 
\begin{equation}
\label{ghost_Jxq_1}
\tilde\Jt_{k,1}(x,q) = \ds\frac{\Jt_{k,2}(x,q)}{1+xq}.
\end{equation}
Equations \eqref{dictionary}, \eqref{def:ghost_others}, \eqref{def:ghost_k}
immediately yield
\begin{equation}
\label{ghost_dictionary}
\B_{(k-1)j+i}=\tilde\Jt_{k,k-i+1}(q^{j},q).
\end{equation}
Hence for $i\in\left\{2,\dots,k-1\right\}$, from~\eqref{coloma:Jtilde},
\begin{align*}
\tilde\Jt_{k,i}(x,q)
= & \ds\sum\limits_{n\geq 0}
(-1)^n\frac{q^{kn^2+(k-i-1)n}x^{(k-1)n}(1-x^{i+1}q^{(2n+1)(i+1)})
(-xq)_n }{(-q)_n(q)_n(xq^{n+1})_\infty(1+xq)}\\
& \mbox{} +  \ds\sum\limits_{n\geq 0}
(-1)^n\frac{q^{1+kn^2+(k-i+1)n}x^{1+(k-1)n}(1-x^{i-1}q^{(2n+1)(i-1)})
(-xq)_n }{(-q)_n(q)_n(xq^{n+1})_\infty(1+xq)}\\
= &\ds\sum\limits_{n\geq 0}
(-1)^n\frac{q^{kn^2+(k-i-1)n}x^{(k-1)n}(-xq)_n}
{(-q)_n(q)_n(xq^{n+1})_\infty(1+xq)}\\
& \mbox{} \cdot \left(1-x^{i+1}q^{(2n+1)(i+1)}+ xq^{2n+1}
(1-x^{i-1}q^{(2n+1)(i-1)})\right).
\end{align*}
Noting that
\begin{align*}
1-x^{i+1}q^{(2n+1)(i+1)} + xq^{2n+1}(1-x^{i-1}q^{(2n+1)(i-1)})&=
1-x^{i+1}q^{(2n+1)(i+1)} + xq^{2n+1}-x^{i}q^{(2n+1)i}\\
& = (1-x^{i}q^{(2n+1)i})(1+xq^{2n+1}),
\end{align*}
we arrive at 
\begin{align}
{\tilde\Jt_{k,i}(x,q)} = \ds\sum\limits_{n\geq 0}
(-1)^n\frac{q^{kn^2+(k-i-1)n}x^{(k-1)n}(1-x^{i}q^{(2n+1)i})
(-xq)_n(1+xq^{2n+1})}{(-q)_n(q)_n(xq^{n+1})_\infty(1+xq)}
\label{ghost_Jxq_upto_2_alternating}
\end{align}
for $i\in\left\{2,\dots,k-1\right\}$.
As for $i=1$,~\eqref{ghost_Jxq_1} gives
\begin{equation*}
\tilde\Jt_{k,1}(x,q) = \ds\sum\limits_{n\geq 0}
(-1)^n\frac{q^{kn^2+(k-2)n}x^{(k-1)n}(1-x^2q^{2(2n+1)})(-xq)_n}
{(-q)_n(q)_n(xq^{n+1})_\infty(1+xq)},
\end{equation*}
which equals \eqref{ghost_Jxq_upto_2_alternating} with $i=1$.

We conclude that 
\begin{equation}\tilde\Jt_{k,i}(x,q)=\ds\sum\limits_{n\geq 0}
(-1)^n\frac{q^{kn^2+(k-i-1)n}x^{(k-1)n}(1-x^{i}q^{(2n+1)i})(-xq)_n(1+xq^{2n+1})}
{(-q)_n(q)_n(xq^{n+1})_\infty(1+xq)}\label{final:ghost_Jxq}
\end{equation}
for $i\in\left\{1,\dots,k-1\right\}$.

\begin{rem}\label{rem:motivated_to_xq}
It is quite straightforward to generalize the reverse-engineering procedure
of this section, i.e., replacing judiciously chosen instances of pure powers 
of $q$ with $x$. 
This can be done throughout the proof of Theorem \ref{thm:main},
and in this way our motivated proof would yield an ``$(x,q)$-proof''
similar in spirit to the ones in Chapter 7 of \cite{A3},
but this time involving two ``$J$'' expressions, namely,
$\tilde{J}$ and $\tilde{\tilde{J}}$. 
A similar observation (without ghosts, of course) can be found in
Section 5 of \cite{AB}, for the original case of the Rogers-Ramanujan
identities.
\end{rem}

\vspace{.3in}

\noindent {\small \sc Department of Mathematics, Rutgers University,
Piscataway, NJ 08854} \\
{\em E-mail address}:
\texttt{skanade@math.rutgers.edu} \\

\noindent {\small \sc Department of Mathematics, Rutgers University,
Piscataway, NJ 08854} \\
{\em E-mail address}:
\texttt{lepowsky@math.rutgers.edu} \\

\noindent {\small \sc Department of Mathematics, Rutgers University,
Piscataway, NJ 08854} \\
{\em E-mail address}:
\texttt{russell2@math.rutgers.edu} \\

\noindent {\small \sc Department of Mathematical Sciences, 
Georgia Southern University, Statesboro, GA 30460} \\
{\em E-mail address}:
\texttt{asills@georgiasouthern.edu} \\

\end{document}